%% file: article.tex
\pdfoutput=1
\documentclass{scrartcl}

\usepackage[T1]{fontenc}
\usepackage[utf8]{inputenc}
\usepackage[english]{babel}
\usepackage{xspace}
\usepackage{lmodern}
\usepackage{exscale}
\usepackage[babel]{microtype}
\usepackage{amsmath, amssymb, mathtools, mathrsfs}
\usepackage{bbm}                % \mathbbm{1}
\usepackage{stmaryrd}           % llbracket, rrbracket
\usepackage{xstring}
\usepackage{mfirstuc}           % makefirstuc
\usepackage{xparse}             % DeclareDocumentCommand
\usepackage{graphicx}
\usepackage{xcolor}
\usepackage{array, booktabs}
\usepackage{import}             % inkscape
\usepackage{todonotes}
\usepackage{mdframed}

\usepackage{tikz}
\usepackage{tikz-cd}
\usetikzlibrary{arrows}

\usepackage[numbered, open=false]{bookmark}
\usepackage{hyperref}
\usepackage{amsthm, thmtools}
\usepackage[compress]{cleveref} % cleveref isn't able to sort properly
\usepackage[xspace]{ellipsis}
\usepackage{caption}            % so that links to figures go to the top of the figure

\hypersetup{
  colorlinks,
  urlcolor = {blue!80!black},
  linkcolor = {red!70!black},
  citecolor = {green!50!black}
}
\numberwithin{figure}{section}
\numberwithin{equation}{section}

\numberwithin{mycounter}{section}
\NewDocumentCommand{\declthm}{m m O{plain}}{% env, name, style
  \declaretheorem[
  name = #2,
  sibling = mycounter,
  style = #3,
  refname = {\MakeLowercase{#2},\MakeLowercase{#2}s},
  Refname = {\makefirstuc{#2},\makefirstuc{#2}s}
  ]
  {#1}
}
\declthm{theorem}{Theorem}
\declthm{proposition}{Proposition}
\declthm{corollary}{Corollary}
\declthm{lemma}{Lemma}
\declthm{definition}{Definition}[definition]
\declthm{example}{Example}[remark]
\declthm{remark}{Remark}[remark]
\declthm{conjecture}{Conjecture}
\declaretheorem[name = Theorem]{theoremintro}

\newcommand{\K}{\Bbbk}
\newcommand{\Q}{\mathbb{Q}}

\NewDocumentCommand{\qiso}{s}{
  \IfBooleanTF #1
  {\xleftarrow{\sim}}
  {\xrightarrow{\sim}}%
}

\newcommand{\operadify}[1]{\mathtt{#1}}
\newcommand{\End}{\operadify{End}}
\newcommand{\Ass}{\operadify{Ass}}
\newcommand{\Ger}{\operadify{Ger}}
\newcommand{\Com}{\operadify{Com}}
\newcommand{\DD}{\operadify{D}}
\newcommand{\SC}{\operadify{SC}}
\newcommand{\Sc}{\operadify{sc}}
\newcommand{\PP}{\operadify{P}}
\newcommand{\QQ}{\operadify{Q}}
\newcommand{\magma}{\operadify{\Omega}}
\newcommand{\tMagma}{{\magma\magma}}
\newcommand{\CoB}{\operadify{CoB}}
\newcommand{\CoP}{\operadify{CoP}}
\newcommand{\PaB}{\operadify{PaB}}
\newcommand{\PaP}{\operadify{PaP}}
\newcommand{\CoPB}{\operadify{CoPB}}
\newcommand{\PaPB}{\operadify{PaPB}}
\newcommand{\CD}{\operadify{CD}}
\newcommand{\Sh}{\operadify{Sh}}
\newcommand{\PaSh}{\operadify{PaSh}}
\newcommand{\PaCD}{\operadify{PaCD}}
\newcommand{\CDh}{\widehat{\CD}}
\newcommand{\PaBh}{\widehat{\PaB}}
\newcommand{\PaCDhp}{\widehat{\PaCD}_{+}}
\newcommand{\PaPCD}{\operadify{PaP}\widehat{\operadify{CD}}_{+}}
\newcommand{\col}[1]{\mathfrak{#1}}

\DeclareMathOperator{\Conf}{Conf}
\DeclareMathOperator{\ob}{ob}
\DeclareMathOperator{\Hom}{Hom}
\DeclareMathOperator{\id}{id}

\newcommand{\categorify}[1]{\mathsf{#1}}
\newcommand{\ZZ}{\mathcal{Z}}  % centre de Drinfeld
\newcommand{\cC}{\categorify{C}}
\newcommand{\cM}{\categorify{M}}
\newcommand{\cN}{\categorify{N}}
\newcommand{\cdga}{\categorify{CDGA}}
\newcommand{\Set}{\categorify{Set}}
\newcommand{\rlz}[1]{\langle #1 \rangle^{\mathbb{L}}}

\newcommand{\bdot}[2]{\draw (#1) node [fill=black, circle, inner
  sep=1.5pt] {} node[above] {#2}}
\newcommand{\wdot}[2]{\draw (#1) node [draw = black, fill=white, circle, inner
  sep=1.5pt] {} node[above] {#2}}
\tikzstyle{cob}=[<->]

\begin{document}

\title{Swiss-Cheese operad and Drinfeld center}
\author{Najib Idrissi\thanks{Laboratoire Paul Painlevé, Université Lille 1 and CNRS, Cité Scientifique, 59655 Villeneuve d'Ascq Cedex, France. \href{mailto:najib.idrissi-kaitouni@math.univ-lille1.fr}{najib.idrissi-kaitouni@math.univ-lille1.fr}}}
\date{January 10, 2017}
\maketitle

\begin{abstract}
  We build a model in groupoids for the Swiss-Cheese operad, based on parenthesized permutations and braids. We relate algebras over this model to the classical description of algebras over the homology of the Swiss-Cheese operad. We extend our model to a rational model for the Swiss-Cheese operad, and we compare it to the model that we would get if the operad Swiss-Cheese were formal.
\end{abstract}

\tableofcontents

\section{Introduction}
\label{sec.introduction}

The little disks operads $\DD_{n}$ of Boardman--Vogt and May~\cite{BoardmanVogt1973,May1972} govern algebras which are associative and (for $n \geq 2$) commutative up to homotopy. For $n=2$, one can see that the fundamental groupoid of $\DD_{2}$ forms an operad $\pi \DD_{2}$ equivalent to an operad in groupoids $\PaB$, called the operad of parenthesized braids, which governs braided monoidal categories~\cite[§I.6]{Fresse2016a}. Since the homotopy of $\DD_{2}$ is concentrated in degrees $\leq 1$, this is enough to recover $\DD_{2}$ up to homotopy. For $n=1$, one can also easily see that $\pi\DD_{1}$ is equivalent to an operad $\PaP$, called the operad of parenthesized permutations, which governs monoidal categories.

The Swiss-Cheese operad $\SC = \SC_{2}$ of Voronov~\cite{Voronov1999} governs the action of a $\DD_{2}$-algebra on a $\DD_{1}$-algebra by a central morphism. As explained in~\cite{Hoefel2009}, the Swiss-Cheese operad is intimately related to the ``Open-Closed Homotopy Algebras'' (OCHAs) of Kajiura and Stasheff~\cite{KajiuraStasheff2006}, which are of great interest in string field theory and deformation quantization.

We aim to study the fundamental groupoid of $\SC$, which is still an operad. This fundamental groupoid is again enough to recover $\SC$ up to homotopy. In a first step, we established the following theorem:

\begin{theoremintro}[{See \Cref{thm.model-sc} and \Cref{cor.alg-papb}}] \label{thm.A} The fundamental groupoid operad $\pi \SC$ is equivalent to an operad $\PaPB$ whose algebras are triples $(\cM, \cN, F)$, where $\cN$ is a monoidal category, $\cM$ is a braided monoidal category, and $F : \cM \to \ZZ(\cN)$ is a strong braided monoidal functor from $\cM$ to the Drinfeld center $\ZZ(\cN)$ of $\cN$.
\end{theoremintro}

In this theorem, the monoidal categories have no unit. We also consider the unitary version $\SC_{+}$ of the Swiss-Cheese operad, and we obtain (\Cref{prop.zigzag-unit}) an extension $\PaPB_{+}$ of the model where the monoidal categories have a strict unit and the functor strictly preserves the unit.

The result of Theorem~\ref{thm.A} is a counterpart for operads in groupoids of statements of~\cite[Proposition 31]{Ginot2015} and~\cite[Example 2.13]{AyalaFrancisTanaka2016} about the $\infty$-category of factorization algebras on the upper half plane.

In a second step, we rely on the result of Theorem~\ref{thm.A} to construct an operad rationally equivalent (in the sense of rational homotopy theory) to the completion of the Swiss-Cheese operad. To this end, we use Drinfeld associators, which we see as morphisms $\PaB_{+} \to \CDh_{+}$, where $\CDh_{+}$ is the completed operad of chord diagrams. The existence of such a rational Drinfeld associator is equivalent to the rational formality of $\DD_{2}$, but the inclusion $\DD_{1} \to \DD_{2}$ is not formal; equivalently, the constant morphism $\PaP_{+} \to \CDh_{+}$ does not factor through a Drinfeld associator. We prove the following theorem:
\begin{theoremintro}[{See \Cref{thm.rat-model}}]
  Given a choice of Drinfeld associator $\phi \in Ass^{1}(\Q)$, there is an operad in groupoids $\PaPCD^{\phi}$ built using chord diagrams, parenthesized permutations, and parenthesized shuffles, which is rationally equivalent to $\pi \SC_{+}$.
\end{theoremintro}

The Swiss-Cheese operad is not formal~\cite{Livernet2015}, thus it cannot be recovered from its homology $H_{*}(\SC)$. We use the splitting of $H_{*}(\SC)$ as a product~\cite{Voronov1999} to build an operad in groupoids $\CDh \times_{+} \PaP$, and we compare it to our rational model of $\SC$.

Independently of the author, Willwacher~\cite{Willwacher2015a} found a different model for the Swiss-Cheese operad in any dimension $n \ge 2$ that uses graph complexes. His model extends Kontsevich's~\cite{Kontsevich1999} quasi-isomorphism $\mathrm{Graphs}_{n} \qiso \Omega^{*}(\DD_{n})$ from the proof of the formality of $E_{n}$, whereas our model extends (after passing to classifying spaces) Tamarkin's~\cite{Tamarkin2003} model $\operatorname{B} \hat{\mathbb{U}} \hat{\mathfrak{p}} = \operatorname{B} \CDh$ of $E_{2}$. Thus, in contrast to Willwacher's model, our own model is related to Drinfeld's original approach to quantization. It would be interesting to compare the two, e.g. as was done by Ševera--Willwacher~\cite{SeveraWillwacher2011} for the little $2$-disks operad.

This paper is organized as follows: in \Cref{sec.background}, we recall some background on the Swiss-Cheese operad and relative operads; in \Cref{sec.perm-braid}, we construct two algebraic models for the Swiss-Cheese operad; in \Cref{sec.drinfeld-center}, we describe what the algebras over these models are, using Drinfeld centers; and in \Cref{sec.chord-diagrams}, we construct a rational model in groupoids for the Swiss-Cheese operad using chords diagrams and Drinfeld associators.

\paragraph{Acknowledgments} I would like to thank Benoit Fresse for multiple helpful discussions about the content of this paper.

\section{Background}
\label{sec.background}

The little $n$-disks operad $\DD_{n}$ is built out of configurations of embeddings of little $n$-disks (whose images have disjoint interiors) in the unit $n$-disk, and operadic composition is given by composition of such embeddings -- see \cite{BoardmanVogt1973, May1972} for precise definitions.

The Swiss-Cheese operad $\SC$ is an operad with two colors, $\col{c}$ and $\col{o}$ (standing for ``closed'' and ``open''). The space of operations $\SC(x_{1}, \dots, x_{n}; \col{c})$ with a closed output is equal to $\DD_{2}(n)$ if $x_{1} = \dots = x_{n} = \col{c}$, and it is empty otherwise. The space $\SC(x_{1}, \dots, x_{n}; \col{o})$ is the space of configurations of embeddings of full disks (corresponding to the color $\col{c}$) and half disks (corresponding to the color $\col{o}$), with disjoint interiors, inside the unit upper half disk (see \Cref{fig.exa-sc} for an example). Composition is again given by composition of embeddings.

\begin{figure}[htbp]
  \centering
  \def\svgwidth{0.5\textwidth}
  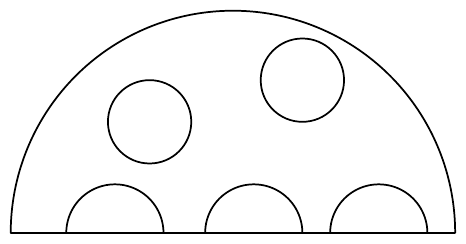
  \caption{Example of an element in $\SC(3,2)$}
  \label{fig.exa-sc}
\end{figure}

The Swiss-Cheese operad is an example of a relative operad~\cite{Voronov1999}: it can be seen as an operad in the category of right modules (in the sense of~\cite{Fresse2009}) over another operad.
\begin{definition}
  Let $\PP$ be a (symmetric, one-colored) operad. A \textbf{relative operad over $\PP$} is an operad $\QQ$ in the category of right modules over $\PP$. Equivalently, it is a two-colored operad $\QQ$ (where the two colors are called $\col{c}$ and $\col{o}$) such that:
  \begin{gather*}
    \QQ(x_{1}, \dots, x_{n}; \col{c}) =
    \begin{cases}
      \PP(n), & \text{if } x_{1} = \dots = x_{n} = \col{c}; \\
      \varnothing, & \text{otherwise}.
    \end{cases}
    \\
    \QQ(n,m) : = \QQ(\underbrace{\col{o}, \dots,
      \col{o}}_{n}, \underbrace{\col{c}, \dots, \col{c}}_{m} ;
    \col{o}) = \bigl( \QQ(n) \bigr)(m).
  \end{gather*}
  If $\QQ$ is such a relative operad, we will write
  $\QQ^{\col{c}}(m) := \PP(m)$.
\end{definition}

The Swiss-Cheese operad $\SC$ is a relative operad over the little disks operad $\DD_{2}$. We also consider the unitary version of the Swiss-Cheese operad $\SC_{+}$, which is a relative operad over the unitary version of the little disks operad $\DD_{2}^{+}$, and which satisfies $\SC_{+}(0,0) = *$. Composition with the nullary elements simply forgets half disks or full disks of the configuration.

\begin{remark}
  We consider a variation of the Swiss-Cheese operad, where we allow operations with only closed inputs and an open output, whereas in Voronov's definition these configurations are forbidden. We write $\SC^{\mathrm{vor}}$ for Voronov's version, so that $\SC^{\mathrm{vor}}(0,m) = \varnothing$ while $\SC(0,m) \simeq \DD_{2}(m) \neq \varnothing$.
\end{remark}

\section{Permutations and braids}
\label{sec.perm-braid}

\subsection{Colored version}
\label{sec.colored-version}

We first define an operad in groupoids $\CoPB$, the \textbf{operad of colored permutations and braids}. It is an operad relative over $\CoB$, the operad of colored braids \cite[§I.5]{Fresse2016a}.

Let $D^{+} = \{ z \in \mathbb{C} \mid \Im z \ge 0, |z| \leq 1 \}$ be the upper half disk, and let
\[ \Conf(n,m) = \{ (z_{1}, \dots, z_{n}, u_{1}, \dots, u_{m}) \in D^{+} \mid \Im z_{i} = 0, \Im u_{j} > 0, z_{i} \neq z_{j}, u_{i} \neq u_{j} \} \]
be the set of configurations of $n$ points on the real interval $[-1,1]$ and $m$ points in the upper half disk.

The disk-center mapping $\omega : \SC(n,m) \qiso \Conf(n,m)$, sending each disk to its center, is a weak equivalence~\cite{Voronov1999}. Let $\Sigma_{k}$ be the $k$th symmetric group, and let $\mathrm{Sh}_{n,m}$ be the set of $(n,m)$-shuffles:
\[ \mathrm{Sh}_{n,m} = \{ \mu \in \Sigma_{n+m} \mid \mu(1) < \dots < \mu(n), \mu(n+1) < \dots < \mu(n+m) \}. \]

For every $\mu \in \mathrm{Sh}_{n,m}$, we choose a configuration $c^{0}_{\mu} \in \Conf(n,m)$ with $n$ ``terrestrial'' points (on the real axis) and $m$ ``aerial'' points (with positive imaginary part), in the left-to-right order given by the $(n,m)$-shuffle $\mu$. For example we can choose:
\begin{equation}
\mu = (14|235) \in \mathrm{Sh}_{2,3} \leadsto c^{0}_{\mu} =
  \begin{gathered}
    \begin{tikzpicture}[scale = 0.4]
      \draw (5,0) arc (0:180:5); \draw (-5,0) -- (5,0);
      \bdot{-4,0}{1}; \bdot{-2,2}{1}; \bdot{0,2}{2}; \bdot{2,0}{2};
      \bdot{3.8,2}{3};
    \end{tikzpicture}
  \end{gathered}
  \in \Conf(2,3),
  \label{eq.exa-ob-copb}
\end{equation}
We consider the set:
\[ C^{0}(n,m) = \{ \sigma \cdot c^{0}_{\mu} \}_{\sigma \in \Sigma_{n} \times \Sigma_{m}, \, \mu \in \mathrm{Sh}_{n,m}} \subset \Conf(n,m), \]
where $\Sigma_{n} \times \Sigma_{m}$ acts by permuting labels.

\begin{example}
  For example, $C^{0}(2,1)$ can be chosen to be:
  \[ C^{0}(2,1) = \left\{
      \begin{gathered}
        \begin{tikzpicture}[scale=0.7]
          \foreach[count=\i] \ya/\la/\yb/\lb/\yc/\lc in
          { 0/1/1/1/1/2, 0/1/1/2/1/1, 1/1/0/1/1/2, 1/2/0/1/1/1, 1/1/1/2/0/1, 1/2/1/1/0/1 }
          {
            \begin{scope}[xshift={5cm*mod(\i-1,3)}, yshift={-3cm*int((\i-1)/3)} ]
              \draw (2,0) arc (0:180:2); \draw (-2,0) -- (2,0);
              \bdot{-1,\ya}{\la}; \bdot{0,\yb}{\lb}; \bdot{1,\yc}{\lc};
            \end{scope}
          }
        \end{tikzpicture}
      \end{gathered}
    \right\} \]
  The precise position of the points does not matter for our purposes, only their left-to-right order.
\end{example}

\begin{definition}
  The groupoid $\CoPB(n,m)$ is the restriction of the fundamental groupoid of $\Conf(n,m)$ to the set $C^{0}(n,m) \subset \Conf(n,m)$ (i.e.\ it is its full subcategory with these objects):
  \[ \CoPB(n,m) := \pi \Conf(n,m)_{\mid C^{0}(n,m)}. \]
\end{definition}

The set $\ob \CoPB(n,m) = C^{0}(n,m)$ is isomorphic to $\mathrm{Sh}_{n,m} \times \Sigma_{n} \times \Sigma_{m}$. We represent these objects by sequences of $n$ ``terrestrial'' points (drawn in white and labeled by $\{1, \dots, n\}$) and $m$ ``aerial'' points (drawn in black and labeled by $\{1,\dots,m\}$) on the interval $I = [-1,1]$; the order in which terrestrial and aerial points appear is given by the shuffle. For example, the element in \Cref{eq.exa-ob-copb} is represented by:
\[ \begin{tikzpicture}
    \draw[|-|] (0,0) -- (6,0);
    \wdot{1,0}{1};
    \bdot{2,0}{1};
    \bdot{3,0}{2};
    \wdot{4,0}{2};
    \bdot{5,0}{3};
  \end{tikzpicture}.
\]

Morphisms between two such configurations are given by isotopy classes of bicolored braids, where strands between terrestrial points never go behind any other strand, including other terrestrial strands (indeed, they represent paths in the interval $[-1,1]$, and points cannot move over one another in $\Conf_{n}([-1,1])$, nor can they go behind the paths in the open upper half disk). See \Cref{fig.exa-cobp} for an example of an element in $\CoPB(2,3)$, and \Cref{fig.comp-copb} for the corresponding path in $\Conf(2,3)$.

\begin{figure}[htbp]
  \begin{minipage}{0.45\linewidth}
    \centering
    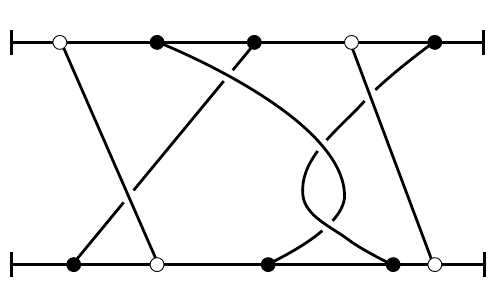
    \caption{Element in $\CoPB(2,3)$}
    \label{fig.exa-cobp}
  \end{minipage}
  \begin{minipage}{0.45\linewidth}
    \centering
    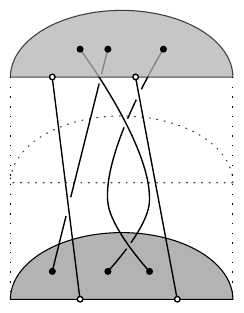
    \caption{Corresponding path in $\Conf(2,3)$}
    \label{fig.comp-copb}
  \end{minipage}
\end{figure}

To not confuse objects of $\CoPB$ and objects of $\CoB$, and to be coherent with the graphical representation of $\tMagma$ in \Cref{sec.magmas}, we draw the objects of $\CoB$ with ends in the shape of chevrons:
\[ \tikz{\draw[cob] (0,0) -- (1,0); \bdot{0.3,0}{1}; \bdot{0.7,0}{2};}
  \in \ob \CoB(2). \]

The symmetric sequence $\CoPB(n) = \{\CoPB(n,m)\}_{m \ge 0}$ is a right module over $\CoB$ by inserting a colored braid in a tubular neighborhood of an aerial strand.  (\Cref{fig.exa-struct-papb-rmod}). Similarly, the operad structure inserts a colored braid in a tubular neighborhood of a terrestrial strand (\Cref{fig.exa-struct-papb-op}). One can easily check that this gives a relative operad over $\CoB$ (in the same manner that one checks that $\CoB$ itself is an operad, cf.\ \cite[§I.5]{Fresse2016a}).

\begin{figure}[htbp]
  \centering
  \[\begin{gathered}
      \def\svgwidth{0.3\textwidth} \import{fig/}{example_copb.pdf_tex}
      \\
      \CoPB(2,3)
    \end{gathered}
    \circ_{1}^{\col{c}}
    \begin{gathered}
      \def\svgwidth{0.1\textwidth} 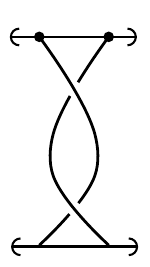
      \\
      \CoB(2)
    \end{gathered}
    =
    \begin{gathered}
      \def\svgwidth{0.3\textwidth} 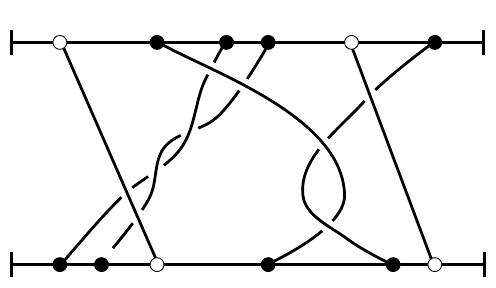
      \\
      \CoPB(2,4)
    \end{gathered}
  \]
  \caption{Definition of the right $\CoB$-module structure}
  \label{fig.exa-struct-papb-rmod}
\end{figure}

\begin{figure}[htbp]
  \centering
  \[\begin{gathered}
      \def\svgwidth{0.3\textwidth} \import{fig/}{example_copb.pdf_tex}
      \\
      \CoPB(2,3)
    \end{gathered}
    \circ_{1}^{\col{o}}
    \begin{gathered}
      \def\svgwidth{0.1\textwidth} 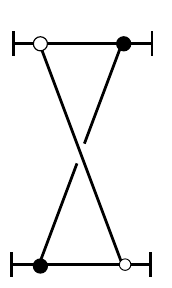 \\
      \CoPB(1,1)
    \end{gathered}
    =
    \begin{gathered}
      \def\svgwidth{0.3\textwidth} 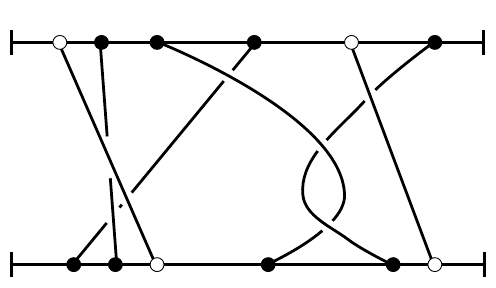 \\
      \CoPB(2,4)
    \end{gathered}
  \]
  \caption{Definition of the operad structure}
  \label{fig.exa-struct-papb-op}
\end{figure}

\subsection{Magmas}
\label{sec.magmas}

\begin{definition}
  Let $\magma$ be the \textbf{magma operad} (as in \cite{Fresse2016a}), defined as the free symmetric operad $\mathbb{O}(\mu_{\col{c}})$ on a single generator $\mu_{\col{c}} = \mu_{\col{c}}(x_{1}, x_{2})$ of arity $2$ (where $\Sigma_{2}$ acts freely on $\mu_{\col{c}}$). Its algebras are sets endowed with a product satisfying no further conditions.
\end{definition}

Elements of $\magma(n)$ are parenthesizations of a permutation of $n$ elements, for example $(((x_{1}x_{3})(x_{2}x_{4}))x_{5}) \in \Omega(5)$. The index $\col{c}$ of $\mu_{\col{c}}$ is there to be coherent with the following definition:

\begin{definition}
  Let $\tMagma = \mathbb{O}(\mu_{\col{c}}, f, \mu_{\col{o}})$ be the free colored operad on the three generators $\mu_{\col{c}} \in \tMagma(\col{c}, \col{c}; \col{c})$, $f \in \tMagma(\col{c}; \col{o})$ et $\mu_{\col{o}} \in \tMagma(\col{o}, \col{o}; \col{o})$. It is a relative operad over $\magma$.
\end{definition}
An algebra over $\tMagma$ is the data of two magmas $M$, $N$, and of a mere function $f : M \to N$ (not necessarily preserving the product).

\begin{lemma}
  The suboperad of $\SC$ generated by the following three elements is free on those generators:
  \begin{align*}
  \mu_{\col{c}}
  & =
    \begin{gathered} \begin{tikzpicture}[scale = 0.5]
      \draw (0,0) circle (2);
      \draw (-1,0) circle (1); \node at (-1,0) {$1$};
      \draw (1,0) circle (1); \node at (1,0) {$2$};
    \end{tikzpicture} \end{gathered},
  & \mu_{\col{o}}
  & =
    \begin{gathered} \begin{tikzpicture}[scale = 0.5]
        \draw (2,0) arc (0:180:2);
        \draw (0,0) arc (0:180:1); \node at (-1,0.5) {$1$};
        \draw (2,0) arc (0:180:1); \node at (1,0.5) {$2$};
        \draw (-2,0) -- (2,0);
      \end{tikzpicture} \end{gathered},
  & f
  & =
    \begin{gathered} \begin{tikzpicture}[scale = 0.5]
        \draw (2,0) arc (0:180:2);
        \draw (0,1) circle (1);
        \node at (0,1) {$1$};
        \draw (-2,0) -- (2,0);
      \end{tikzpicture} \end{gathered}.
  \end{align*}
\end{lemma}
\begin{proof}
  We would like to show that the induced morphism $i : \tMagma \to \SC$, sending the three generators of $\tMagma$ to the elements depicted in the lemma, is an embedding, i.e.\ an isomorphism onto its image. The image of this induced morphism is by definition the suboperad generated by the three elements, hence the lemma.

  The fact that the suboperad of $\DD_{2}$ generated by $\mu_{\col{c}}$ is free is given by \cite[Proposition I.6.2.2(a)]{Fresse2016a}. Let $\alpha \in \SC(n,m)$ be a configuration, as in \Cref{fig.exa-magma}. We will build an element of $\tMagma$ which is sent to $\alpha$ under $i$. This set-level retraction is not necessarily a morphism of operads but still shows that $i$ is injective, which will prove the lemma.

  \begin{figure}[htbp]
    \centering
    \def\svgwidth{0.8\textwidth}
    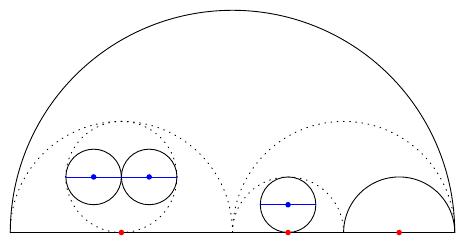
    \caption{Example element of $\tMagma(1,3)$}
    \label{fig.exa-magma}
  \end{figure}

  We first regroup the $m$ full disks into connected components $C_{1}, \dots, C_{r}$. For each $C_{i}$, we consider the center of the middle horizontal interval (in blue on the figure), which we project onto the real line (in red). These points, together with the centers of the $n$ half disks, make up a dyadic configuration on the horizontal diameter of the ambient half disk. By \cite[Proposition I.6.2.2(a)]{Fresse2016a}, such a dyadic configuration is equivalent to an element $u \in \magma(n+r)$ (which we see as an iterate of $\mu_{\col{c}}$).

  For each $C_{i}$ (corresponding to an input of $u$), we apply the same proposition \cite[Proposition I.6.2.2(a)]{Fresse2016a} to get an element $v_{i} \in \magma(k_{i})$ (which we see as a iterate of $\mu_{\col{o}}$), and $\sum k_{i} = m$. If we plug $f(v_{i})$ in the corresponding inputs of $u$, we get an element of $\tMagma(n,m)$, and by construction this elements gets sent to $\alpha$ by $i$.
\end{proof}

We consider the following graphical representation for elements of $\tMagma$ with open output, where the generators are represented as follows:
\begin{align*}
  \mu_{\col{c}}
  & \leadsto \tikz{\draw[cob] (0,0) -- (2,0); \draw[|-] (1,0) -- (2,0);
    \bdot{0.5,0}{1}; \bdot{1.5,0}{2};},
  & \mu_{\col o}
  & \leadsto \tikz{\draw[|-|] (0,0) -- (2,0); \draw[|-] (1,0) -- (2,0);
    \wdot{0.5,0}{1}; \wdot{1.5,0}{2};},
  & f
  & \leadsto \tikz{\draw[|-|] (0,0) -- (1,0); \draw[cob] (0.1,0) --
    (0.9,0); \bdot{0.5,0}{1};}.
\end{align*}

For example, this is the representation of the element of \Cref{fig.exa-magma}:
\[ \mu_{\col{o}}(f(\mu_{\col{c}}(x_{1}, x_{2})), \mu_{\col{o}}(f(x_{3}), y_{1}))
  \leadsto
  \begin{tikzpicture}
      \draw[|-|] (0,0) -- (4,0);
      \draw[|-|] (1,0) -- (2,0);
      \draw[cob] (0.1,0) -- (1.9,0);
      \bdot{0.5,0}{1}; \bdot{1.5,0}{2};
      \draw[|-] (3,0) -- (4,0);
      \draw[cob] (2.1,0) -- (2.9,0);
      \bdot{2.5,0}{3};
      \wdot{3.5,0}{1};
    \end{tikzpicture} \in \tMagma(1,3)
\]
Each $\mu_{\col{o}}$ is represented by cutting in half the interval; $f$ is represented by parentheses, and $\mu_{\col{c}}$ is again represented by cutting the interval inside the parentheses in half. Closed inputs are represented by black points, while open inputs are represented by white points.

\begin{remark}
  The parentheses separating the aerial points are really necessary in the representation. For example these are two different objects:
  \begin{gather*}
    f(\mu_{\col c}(x_{1},x_{2})) =
    \begin{tikzpicture}
      \draw[|-|] (0,0) -- (2,0); \draw[cob] (0.1,0) -- (1.9,0);
      \draw[|-] (1,0) -- (2,0);
      \bdot{0.5,0}{1}; \bdot{1.5,0}{2};
    \end{tikzpicture}
    =
    \begin{gathered}
      \begin{tikzpicture}[scale=0.4]
        \draw[dashed] (0,0) circle (2);
        \draw (-1,0) circle (1); \node at (-1,0) {$1$};
        \draw (1,0) circle (1); \node at (1,0) {$2$};
        \draw (-4,-2) -- (4,-2); \draw (4,-2) arc (0:180:4);
      \end{tikzpicture}
    \end{gathered}
    \\
    \mu_{\col o}(f(x_{1}), f(x_{2}))
    =
    \begin{tikzpicture}
      \draw[|-|] (0,0) -- (2,0);
      \draw[cob] (0.1,0) -- (0.9,0);  \draw[cob] (1.1,0) -- (1.9,0);
      \draw[|-] (1,0) -- (2,0);
      \bdot{0.5,0}{1}; \bdot{1.5,0}{2};
    \end{tikzpicture}
    =
    \begin{gathered}
      \begin{tikzpicture}[scale=0.8]
        \draw (2,0) arc (0:180:2);
        \draw[dashed] (0,0) arc (0:180:1); \draw (-1,0.5) circle (0.5);
        \node at (-1,0.5) {$1$};
        \draw[dashed] (2,0) arc (0:180:1); \draw (1,0.5) circle (0.5);
        \node at (1,0.5) {$2$};
        \draw (-2,0) -- (2,0);
      \end{tikzpicture}
    \end{gathered}
  \end{gather*}
\end{remark}

\subsection{Parenthesized version}
\label{sec.parenth-vers}

We will now define $\PaPB$, the \textbf{operad of parenthesized permutations and braids}, a relative operad over $\PaB$. The definition of $\PaPB$ is given as a pullback of $\CoPB$, similarly to how $\PaB$ is a pullback of $\CoB$.

\begin{definition}\label{def.papb}
  We consider the morphism $\omega : \tMagma \to \ob \CoPB$, given on generators by:
  \begin{align*}
    \mu_{\col{c}} & \mapsto \tikz{\draw[cob] (0,0) -- (1,0); \bdot{0.3,0}{1}; \bdot{0.7,0}{2};},
    & f & \mapsto \tikz{\draw[|-|] (0,0) -- (1,0); \bdot{0.5,0}{1};},
    & \mu_{\col{o}} & \mapsto \tikz{\draw[|-|] (0,0) -- (1,0); \wdot{0.3,0}{1}; \wdot{0.7,0}{2};},
  \end{align*}
  and we define $\PaPB := \omega^{*} \CoPB$, the pullback of $\CoPB$ along $\omega$. It is an operad in groupoids such that $\ob \PaPB = \tMagma$ and
  \[ \Hom_{\PaPB(n,m)}(u,v) := \Hom_{\CoPB(n,m)}(\omega(u), \omega(v)) \]
  for $u,v \in \tMagma(n,m)$.
\end{definition}

\begin{definition}
  A \textbf{categorical equivalence} is a morphism of operads in groupoids which is an equivalence of categories in each arity. Two operads $\PP$ and $\QQ$ are said to be \textbf{categorically equivalent} (and we write $\PP \simeq \QQ$) if they can be connected by a zigzag of categorical equivalences:
  \[ \PP \qiso* \cdot \qiso \dots \qiso* \cdot \qiso \QQ. \]
\end{definition}

Recall that the fundamental groupoid functor $\pi : (\mathsf{Top}, \times) \to (\mathsf{Gpd}, \times)$ is monoidal, thus the fundamental groupoid of a topological operad is an operad in groupoids.

\begin{remark}
  Since each arity of the operad $\SC$ has homotopy concentrated in degrees $\le 1$, it follows that its fundamental groupoid is enough to recover the homotopy type of the operad through the classifying space construction: $\SC \qiso \operatorname{B} \pi \SC$.
\end{remark}

\begin{theorem} \label{thm.model-sc} The operad $\PaPB$ is isomorphic to the fundamental groupoid of $\SC$ restricted to the image of $\tMagma \hookrightarrow \SC$, and we get a zigzag of categorical equivalences:
  \[ \pi \SC \qiso* \pi \SC_{\mid \tMagma} \cong \PaPB \qiso \CoPB. \]
\end{theorem}
\begin{proof}
  The proof of the first part of the proposition is a direct adaptation of the proof of \cite[Proposition I.6.2.2(b)]{Fresse2016a}. We note that $\tMagma \subset \ob \pi \SC$ is a suboperad, thus $\pi \SC_{\mid \tMagma}$ is also a suboperad of $\pi \SC$. For the second part, we note that $\tMagma(n,m)$ meets all the connected components of $\SC(n,m) \sim \Sigma_{n} \times \DD_{2}(m)$, so the first inclusion is a categorical equivalence. Since $\omega : \tMagma \to \ob \CoPB$ is surjective, the second morphism is also a categorical equivalence.
\end{proof}

\section{Drinfeld center}
\label{sec.drinfeld-center}

\subsection{Algebras over \texorpdfstring{$\PaPB$}{PaPB}}
\label{sec.algebras-over-papb}

\begin{definition}
  Let $\cC$ be a (non-unitary) monoidal category. Its suspension $\Sigma \cC$ is a bicategory with a single object. The \textbf{Drinfeld center}~\cite{Majid1991, JoyalStreet1991} of $\cC$ is the braided monoidal category $\ZZ(\cC) = \operatorname{End}(\id_{\Sigma \cC})$. Explicitly, it is given as follows:
  \begin{itemize}
  \item Objects are pairs $(X, \Psi)$, where $X$ is an object of $\cC$ and $\Psi : (X \otimes -) \to (- \otimes X)$ is a \textbf{half-braiding}, i.e.\ a natural isomorphism such that for all $Y,Z \in \cC$ the following diagram commutes:
    \[\begin{tikzcd}[row sep = tiny]
        {} & X \otimes (Y \otimes Z) \rar{\Psi_{Y \otimes Z}} & (Y
        \otimes Z) \otimes X \drar{\alpha} \\
        (X \otimes Y) \otimes Z \urar{\alpha} \drar[swap]{\Psi_{Y}
          \otimes 1} & {} & {} & Y \otimes (Z \otimes X) \\
        {} & (Y \otimes X) \otimes Z \rar[swap]{\alpha} & Y \otimes (X
        \otimes Z) \urar[swap]{1 \otimes \Psi_{Z}}
      \end{tikzcd}\]
  \item Morphisms between $(X, \Psi)$ and $(Y, \Psi')$ are morphisms $f : X \to Y$ of $\cC$ such that, for all $Z \in \cC$, the following diagram commutes
    \[\begin{tikzcd}
        X \otimes Z \rar{f \otimes 1} \dar[swap]{\Psi_{Z}}
        & Y \otimes Z \dar{\Psi'_{Z}} \\
        Z \otimes X \rar[swap]{1 \otimes f}
        & Z \otimes Y
      \end{tikzcd}\]
  \item The tensor product of two objects $(X, \Psi) \otimes (X', \Psi')$ is given by $(X \otimes X', \Psi'')$, where $\Psi''_{Z}$ is defined by the following diagram (that can be rearranged as an hexagon by inverting both vertical $\alpha$'s, see \Cref{fig.f-monoid-center}):
    \[\begin{tikzcd}
        (X \otimes X') \otimes Y \ar{rrr}{\Psi''_{Y}}
        \dar[swap]{\alpha}
        &&& Y \otimes (X \otimes X') \\
        X \otimes (X' \otimes Y) \rar[swap]{1 \otimes \Psi'_{Y}} & X
        \otimes (Y \otimes X') \rar[swap]{\alpha^{-1}} & (X \otimes Y)
        \otimes X' \rar[swap]{\Psi_{Y} \otimes 1} & (Y \otimes X)
        \otimes X' \uar[swap]{\alpha}
      \end{tikzcd}\]
  \item The braiding $(X, \Psi) \otimes (X', \Psi') \to (X', \Psi') \otimes (X, \Psi)$ is given by $\Psi_{X'}$ and the associator is given by the associator of $\cC$.
  \end{itemize}
\end{definition}

We consider the following elements of $\PaPB$:
\begin{center}
  \begin{tabular}{cccc}
    \toprule
    $\mu_{\col{c}} \in \ob \PaB(2)$ & $\mu_{\col{o}} \in \ob \PaPB(2,0)$ & $f \in \ob \PaPB(0,1)$ & $\tau \in \PaB(2)$
    \\ \midrule
    \tikz[baseline=-1cm]{\draw[<-|] (0,0) -- (0.5,0);
    \draw[->] (0.5,0) -- (1,0);
    \bdot{0.25,0}{1};
    \bdot{0.75,0}{2};}
    & \tikz[baseline=-1cm]{\draw[|-|] (0,0) -- (0.5,0); \draw[-|] (0.5,0) --
      (1,0); \wdot{0.25,0}{1}; \wdot{0.75,0}{2};}
    & \tikz[baseline=-1cm]{\draw[|-|] (0,0) -- (1,0); \draw[cob] (0.1,0) -- (0.9,0);
      \bdot{0.5,0}{1};}
    & \def\svgwidth{0.12\textwidth}
      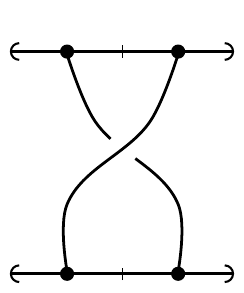
    \\ \midrule
    $p \in \PaPB(0,2)$ & $\psi \in \PaPB(1,1)$ & $\alpha_{\col{c}} \in \PaB(3)$ & $\alpha_{\col{o}} \in \PaPB(3,0)$
    \\ \midrule
    \def\svgwidth{0.12\textwidth}
      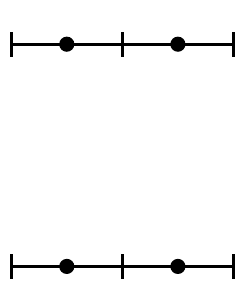
    & \def\svgwidth{0.12\textwidth}
      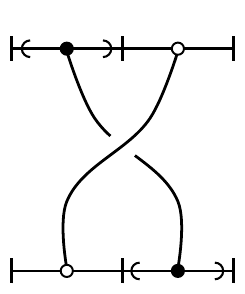
    & \def\svgwidth{0.25\textwidth}
      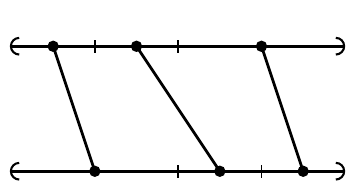
    & \def\svgwidth{0.25\textwidth}
      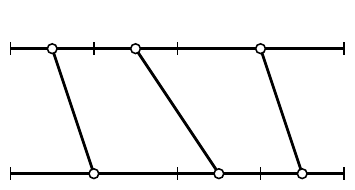
    \\ \bottomrule
  \end{tabular}
\end{center}

\begin{theorem}\label{thm.gen-rel-papb}
  Let $\PP$ be a $\{\col c, \col o\}$-colored operad\footnote{The operad will not necessarily be a relative operad, but we will still use the notation $\PP(n,m) = \PP(\col{c}^{m}, \col{o}^{n}; \col{o})$ and $\PP^{\col{c}}(m) = \PP(\col{c}^{m}; \col{c})$.} in the category of categories, let $m_{\col{c}} \in \ob \PP^{\col{c}}(2)$, $m_{\col{o}} \in \ob \PP(2,0)$, $F \in \ob \PP(0,1)$ be objects, and let
  \begin{align*}
    a_{\col{c}} & : m_{\col{c}}(m_{\col{c}}(x_{1}, x_{2}), x_{3}) \to m_{\col{c}}(x_{1},
    m_{\col{c}}(x_{2}, x_{3})),
    & \pi & : m_{\col{o}}(f(x_{1}), f(x_{2})) \to f(m_{\col{c}}(x_{1}, x_{2})), \\
    t & : m_{\col{c}}(x_{1}, x_{2}) \to m_{\col{c}}(x_{2}, x_{1}),
    & \Psi & : m_{\col{o}}(f(x_{1}), y_{1}) \to m_{\col{o}}(y_{1}, f(x_{1})), \\
    a_{\col{o}} & : m_{\col{o}}(m_{\col{o}}(y_{1}, y_{2}), y_{3}) \to m_{\col{o}}(y_{1},
    m_{\col{o}}(y_{2}, y_{3})),
  \end{align*}
  be isomorphisms. Then there exists a morphism $\theta : \PaPB \to \PP$ such that
  \begin{align*}
    \theta(\mu_{\col{c}}) & = m_{\col{c}}, & \theta(\mu_{\col{o}}) & = m_{\col{o}}, & \theta(f) & = F, & \theta(\alpha_{\col{c}}) & = a_{\col{c}}, \\
    \theta(\alpha_{\col{o}}) & = a_{\col{o}}, & \theta(\tau) & = t, & \theta(p) & = \pi, & \theta(\psi) & = \Psi,
  \end{align*}
  (in which case this morphism is unique) if, and only if, the coherence diagrams of \Cref{fig.pentagons,fig.hexagons,fig.f-monoidal,fig.f-center,fig.f-braided,fig.f-monoid-center} commute.
\end{theorem}

\begin{figure}[htbp]
  \centering
  \begin{tikzcd}[row sep = tiny, column sep = tiny, font = \small]
     {} & m(m(m(x_1, x_2), x_3), x_4) \dlar[swap]{m(a,\id)} \ar{ddr}{a} \\
     m(m(x_1, m(x_2, x_3)), x_4) \ar[swap]{dd}{a} & {} \\
     {} & {} & m(m(x_1, x_2), m(x_3, x_4)) \ar{ddl}{a} \\
     m(x_1, m(m(x_2, x_3), x_4)) \ar[swap]{dr}{m(\id,a)} \\
     {} & m(x_1, m(x_2, m(x_3, x_4)))
  \end{tikzcd}
  \caption{Pentagon for $(m, a) = (m_{\col{c}}, a_{\col{c}})$ and $(m_{\col{o}}, a_{\col{o}})$}
  \label{fig.pentagons}
\end{figure}
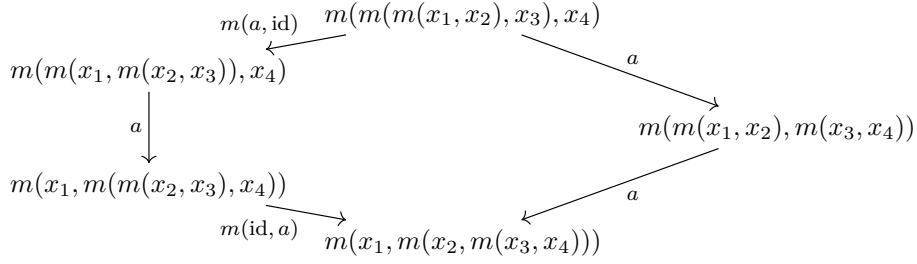

% cramped =
% every matrix/.append style={inner sep=+-0.3em}, every cell/.append style={inner sep=+0.3em}
\begin{figure}[htbp]
  \centering
  \begin{tikzcd}[row sep = tiny, column sep = tiny, every matrix/.append style={inner sep=+-0.3em}, every cell/.append style={inner sep=+0.3em}]
    {} & m_{\col{c}}(m_{\col{c}}(x_1, x_2), x_3) \dlar[swap]{m_{\col{c}}(t,\id)} \drar{a_{\col{c}}} & {} \\
    m_{\col{c}}(m_{\col{c}}(x_2, x_1), x_3) \dar[swap]{a_{\col{c}}} & {} & m_{\col{c}}(x_1, m_{\col{c}}(x_2, x_3)) \dar{t} \\
    m_{\col{c}}(x_2, m_{\col{c}}(x_1, x_3)) \drar[swap]{m_{\col{c}}(\id,t)} & {} & m_{\col{c}}(m_{\col{c}}(x_2, x_3), x_1) \dlar[swap]{a_{\col{c}}} \\
    {} & m_{\col{c}}(x_2, m_{\col{c}}(x_3, x_1))
  \end{tikzcd}
  \begin{tikzcd}[row sep = tiny, column sep = tiny, every matrix/.append style={inner sep=+-0.3em}, every cell/.append style={inner sep=+0.3em}]
    {} & m_{\col{c}}(x_1, m_{\col{c}}(x_2, x_3)) \dlar{m_{\col{c}}(\id,t)} \drar{a_{\col{c}}^{-1}} \\
    m_{\col{c}}(x_1, m_{\col{c}}(x_3, x_2)) \dar[swap]{a_{\col{c}}^{-1}} & {} & m_{\col{c}}(m_{\col{c}}(x_1, x_2), x_3) \dar{t} \\
    m_{\col{c}}(m_{\col{c}}(x_1, x_3), x_2) \drar[swap]{m_{\col{c}}(t,\id)} & {} & m_{\col{c}}(x_3, m_{\col{c}}(x_1, x_2)) \dlar{a_{\col{c}}^{-1}} \\
    {} & m_{\col{c}}(m_{\col{c}}(x_3, x_1), x_2)
  \end{tikzcd}
  \caption{Hexagons}
  \label{fig.hexagons}
\end{figure}

\begin{figure}[htbp]
  \centering
  \begin{tikzcd}[row sep = small, column sep = tiny, every matrix/.append style={inner sep=+-0.3em}, every cell/.append style={inner sep=+0.3em}]
    {} & m_{\col{o}}(m_{\col{o}}(f(x_1), f(x_2)), f(x_3)) \dlar[swap]{m_{\col{o}}(\pi,\id)} \drar{a_{\col{c}}} \\
    m_{\col{o}}(f(m_{\col{c}}(x_1,x_2)), f(x_3)) \dar[swap]{\pi} & {} & m_{\col{o}}(f(x_1),
    m_{\col{o}}(f(x_2), f(x_3))) \dar{m_{\col{o}}(\id, \pi)} \\
    f(m_{\col{c}}(m_{\col{c}}(x_1, x_2), x_3)) \drar[swap]{f(a_{\col{c}})} & {} & m_{\col{o}}(f(x_1),
    f(m_{\col{c}}(x_2, x_3))) \dlar{\pi} \\
    {} & f(m_{\col{c}}(x_1, m_{\col{c}}(x_2, x_3)))
  \end{tikzcd}
  \caption{$F$ is monoidal}
  \label{fig.f-monoidal}
\end{figure}

\begin{figure}[htbp]
  \centering
  \begin{tikzcd}[row sep = small, column sep = tiny, every matrix/.append style={inner sep=+-0.3em}, every cell/.append style={inner sep=+0.3em}]
    {} & m_{\col{o}}(m_{\col{o}}(f(x_1), x_2), x_3) \dlar[swap]{m_{\col{o}}(\Psi,\id)} \drar{a_{\col{o}}} \\
    m_{\col{o}}(m_{\col{o}}(x_2, f(x_1)), x_3) \dar[swap]{a_{\col o}} & {} & m_{\col{o}}(f(x_1), m_{\col{o}}(x_2,x_3)) \dar{\Psi} \\
    m_{\col{o}}(x_2, m_{\col{o}}(f(x_1), x_3)) \drar[swap]{m_{\col{o}}(\id,\Psi)} & {} &
    m_{\col{o}}(m_{\col{o}}(x_2,x_3), f(x_1)) \dlar{a_{\col o}} \\
    {} & m_{\col{o}}(x_2, m_{\col{o}}(x_3, f(x_1)))
  \end{tikzcd}
  \caption{$\Psi$ is a half-braiding}
  \label{fig.f-center}
\end{figure}

\begin{figure}[htbp]
  \centering
  \begin{tikzcd}[every matrix/.append style={inner sep=+-0.3em}, every cell/.append style={inner sep=+0.3em}]
    m_{\col{o}}(f(x_1), f(x_2)) \rar{\Psi} \dar[swap]{\pi} & m_{\col{o}}(f(x_2), f(x_1))
    \dar{\pi} \\
    f(m_{\col{c}}(x_1, x_2)) \rar[swap]{f(t)} & f(m_{\col{c}}(x_2, x_1))
  \end{tikzcd}
  \caption{$F$ is braided}
  \label{fig.f-braided}
\end{figure}

\begin{figure}[htbp]
  \centering
  \begin{tikzcd}[every matrix/.append style={inner sep=+-0.3em}, every cell/.append style={inner sep=+0.3em}, row sep = small, column sep = tiny]
    {} & m_{\col{o}}(f(x_1), m_{\col{o}}(f(x_2), x_3)) \dlar[swap]{m(\id, \Psi)}
    \drar{a_{\col{o}}^{-1}} \\
    m_{\col{o}}(f(x_1), m_{\col{o}}(x_3, f(x_2))) \ar[swap]{ddd}{a_{\col{o}}^{-1}} & {} &
    m_{\col{o}}(m_{\col{o}}(f(x_1), f(x_2)), x_3) \dar{m_{\col{o}}(\pi,\id)} \\
    {} & {} & m_{\col{o}}(f(m_{\col{c}}(x_1,x_2)), x_3) \dar{\Psi} \\
    {} & {} & m_{\col{o}}(x_3, f(m_{\col{c}}(x_1,x_2))) \dar{m_{\col{o}}(\id, \pi^{-1})} \\
    m_{\col{o}}(m_{\col{o}}(f(x_1), x_3), f(x_2)) \drar[swap]{m_{\col{o}}(\Psi,\id)} & {} & m_{\col{o}}(x_3,
    m_{\col{o}}(f(x_1), f(x_2))) \dlar{a_{\col{o}}^{-1}} \\
    {} & m_{\col{o}}(m_{\col{o}}(x_3, f(x_1)), f(x_2))
  \end{tikzcd}
  \caption{$F$ is monoidal w.r.t.\ half-braidings}
  \label{fig.f-monoid-center}
\end{figure}

Recall that an algebra over a colored operad $\QQ$ is a morphism $\QQ \to \End_{(A,B)}$, where
\begin{align*}
  \End_{(A,B)}^{\col c}(n,m)
  & = \hom(B^{\otimes n} \otimes A^{\otimes m}, A),
  & \End_{(A,B)}^{\col o}(n,m)
  & = \hom(B^{\otimes n} \otimes A^{\otimes m}, B).
\end{align*}

Given a morphism $\PaPB \to \End_{(\cM,\cN)}$ with the names as in \Cref{thm.gen-rel-papb} for the images of the generators, the previous coherence diagrams are exactly the diagrams encoding the fact that $(m_{\col{c}}, a_{\col{c}}, t_{\col{c}})$ is a braided monoidal structure on $\cM$, $(m_{\col{o}}, a_{\col{o}})$ is a monoidal structure on $\cN$, and $F$ is a braided monoidal functor to the Drinfeld center. We thus get:
\begin{corollary}\label{cor.alg-papb}
  An algebra over $\PaPB$ is the data of:
  \begin{itemize}
  \item A (non-unitary) monoidal category $(\cN, \otimes)$;
  \item A (non-unitary) braided monoidal category $(\cM, \otimes, \tau)$;
  \item A strong braided monoidal functor $F : \cM \to \ZZ(\cN)$.
  \end{itemize}
\end{corollary}

\begin{definition} \label{def.shuffle-type}
  Between two objects $x,y \in \PaPB(n,m)$ such that the terrestrial (resp.\ aerial) points of $x$ are numbered in the same order as the terrestrial (resp.\ aerial) points $y$, there is a unique morphism $\mu \in \hom_{\PaPB(n,m)}(x,y)$, called a \textbf{shuffle-type morphism}, such that the aerial strands do not cross each other (see \Cref{fig.exa-shuffle-type}).
\end{definition}

\begin{figure}[htbp]
  \centering
  \def\svgwidth{0.6\textwidth} 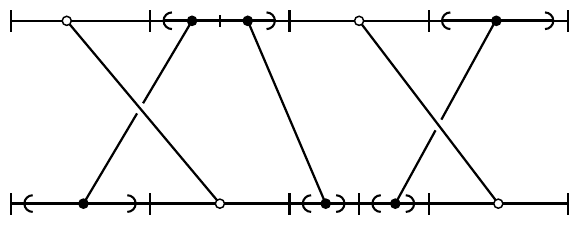
  \caption{Example of shuffle-type morphism}
  \label{fig.exa-shuffle-type}
\end{figure}

\begin{proof}[Proof of \Cref{thm.gen-rel-papb}]
  It is clear (a simple exercise in drawing braid diagrams) that the morphisms of $\PaPB$ satisfy the corresponding relations, thus we get the ``only if'' part of the theorem.

  Let $Y \in \hom_{\PaPB(n,m)}(x_{1}, x_{2})$ be a morphism. We want to decompose it as in \Cref{fig.decomp-papb}:
  \begin{itemize}
  \item We first arbitrarily choose two objects $x_{1}', x_{2}'$ which are in the image of $\PaP(n) \times \PaB(m)$ by $\mu_{\col{o}}(-, f(-))$. In other words, $x_{i}' = \mu_{\col{o}}(x_{i}^{\col{o}}, f(x_{i}^{\col{c}}))$ is the concatenation of an object $x_{i}^{\col{o}} \in \PaP(n) = \PaPB(n,0)$ and of the image by $f$ of an object $x_{i}^{\col{c}} \in \PaB(m)$. We also require that the aerial points (resp.\ the terrestrial points) of $x_{i}'$ are numbered in the same order as those of $x_{i}$.
  \item We take the unique shuffle-type morphism $\mu : x_{1} \to x_{1}'$.
  \item We build a morphism $X = \mu_{\col{o}}(X^{\col{o}}, f(X^{\col{c}})) : x_{1}' \to x_{2}'$. It is the concatenation of $X^{\col{o}} \in \PaP(n)$, and $X^{\col{c}} \in \PaB(m)$.  Explicitly, $X^{\col{o}}$ is the colored permutation where all the aerial strands of $\omega_{*}(Y)$ have been forgotten, and $X^{\col{c}}$ is the colored braid where all the terrestrial strands of $\omega_{*}(Y)$ have been forgotten.
  \item Finally, we take the unique shuffle-type morphism $\mu' : x_{2}' \to x_{2}$.
  \end{itemize}

  \begin{figure}[htbp]
    \centering
    \def\svgwidth{1\textwidth}
    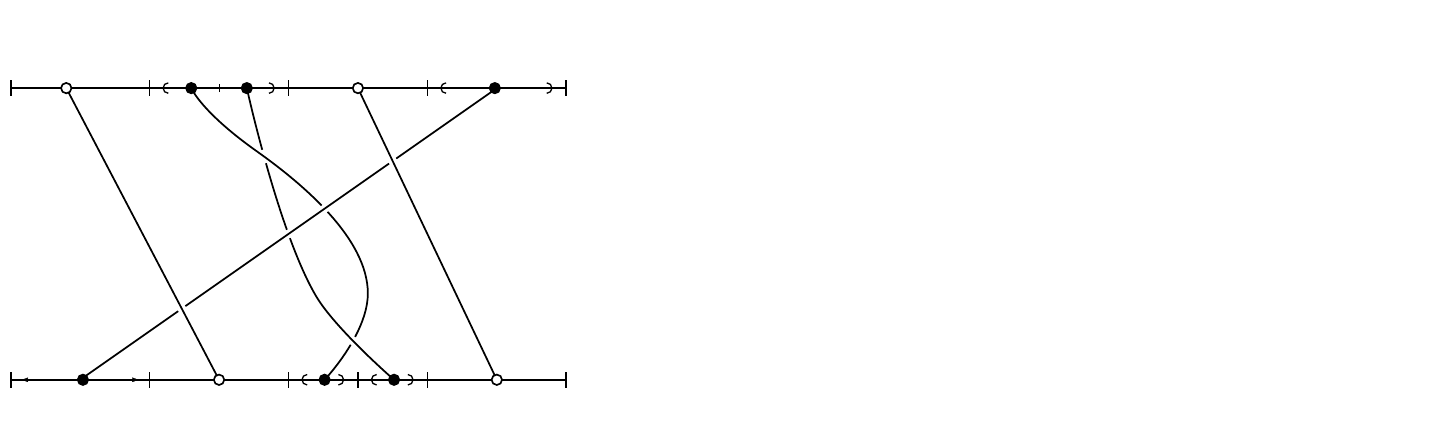
    \caption{Decomposition in $\PaPB(2,3)$}
    \label{fig.decomp-papb}
  \end{figure}

  By construction, $Y = \mu' \circ X \circ \mu$. Besides, this decomposition is unique given the specified intermediary objects $x_{1}'$, $x_{2}'$, so it suffices to show that $\theta$ can be defined unequivocally on each part, that it doesn't depend on the choice of $x_{1}'$ and $x_{2}'$, and that it is compatible with operadic composition.

  The shuffle-type morphisms are all in the suboperad of $\PaPB$ generated by $\alpha_{\col{o}}^{\pm 1}$, $p^{\pm 1}$, $\psi^{\pm 1}$: first one can cut the objects of $\PaB$ in the smallest possible pieces with $p^{-1}$, the $\alpha_{\col{o}}^{\pm 1}$ and $\psi^{\pm 1}$ can be used to bring all the aerial points at their positions, and finally $p$ is used to glue back all the aerial pieces. By the theorems I.6.1.7 and I.6.2.4 of \cite{Fresse2016a}, the two morphisms $X^{\col{o}} \in \PaP(n)$ and $X^{\col{c}} \in \PaB(m)$ are respectively in the suboperads generated by $\mu_{\col{o}}^{\pm 1}, \alpha_{\col{o}}^{\pm 1}$ and by $\mu_{\col{c}}^{\pm 1}, \alpha_{\col{c}}^{\pm 1}, \tau^{\pm 1}$. It follows that every morphism $Y$ of $\PaPB$ is in the suboperad generated by all these elements, thus the morphism $\theta : \PaPB \to \PP$, if it exists, is unique.

  By the same theorems of~\cite{Fresse2016a}, the pentagons (\Cref{fig.pentagons}) and the hexagons (\Cref{fig.hexagons}) show that the morphism $\theta$ can be defined with no ambiguity on the two pieces $X^{\col{c}}$ and $X^{\col{o}}$. The possible choices for $x_{1}'$ and $x_{2}'$ are all related by associators, so the pentagons (\Cref{fig.pentagons}) and MacLane's coherence theorem~\cite{Mac1998} for monoidal categories show that the image does not depend on the choice of $x_{1}'$ and $x_{2}'$.

  Let $\mu \in \PaPB(n,m)$ be a shuffle-type morphism; we saw that it could be decomposed in terms of $p^{\pm 1}$, $\psi^{\pm 1}$ and $\alpha^{\pm 1}$. The coherence theorem of MacLane~\cite{Mac1998} and the coherence theorem of Epstein~\cite{Epstein1966} on monoidal functors (non-symmetric version) show that, thanks to the pentagons (\Cref{fig.pentagons}) and the fact that $F$ is monoidal (\Cref{fig.f-monoidal}), the image $\theta(\mu)$ neither depends on the choice of associator decomposition, nor on the choice of decomposition of $p^{\pm 1}$, nor on the way the $\psi^{\pm 1}$ are gathered in the parenthesizing. It thus suffices to define $\theta$ on the underlying morphism of $\CoPB$.

  This last morphism is actually an element of the braid group $B_{n+m}$ (of course not all the elements of the braid group can given a morphism: terrestrial strands cannot cross any other strand). By seeing $\psi$ as a braiding, and by interpreting the relations of \Cref{fig.f-center,fig.f-monoid-center} as two hexagon relations, we can adapt the proof of the step 2 of \cite[Theorem 6.2.4]{Fresse2016a} to see that the image by $\theta$ of this braid does not depend on its representation. Finally, $\theta$ is well-defined on every morphism. It remains to show that it respects operad composition.

  By adapting the fourth step of the proof of the same theorem of \cite{Fresse2016a} and by using the relation of \Cref{fig.f-braided}, we can see that $\theta$ respects operadic composition. Indeed, shuffle-type morphism are sent by construction on elements decomposed in terms of associators, their inverses, $p^{\pm 1}$ and $\psi^{\pm 1}$, while for example $\psi \circ_{1}^{\col{o}} \id_{f} = f(\tau)$; but thanks to the relation of \Cref{fig.f-braided}, both elements are equal in the image.
\end{proof}

By dropping all mentions of parenthesizing, we get:
\begin{proposition}
  An algebra over $\CoPB$ consists of a strict (non-unitary) monoidal category $\cN$, a strict braided (non-unitary) monoidal category $\cM$, and of a strict braided monoidal functor $F : \cM \to \ZZ(\cN)$.
\end{proposition}

\subsection{Unitary versions}
\label{sec.unitary-versions}

We are going to define unitary versions $\CoPB_{+}$ and $\PaPB_{+}$ of the operads we are studying, satisfying $\CoPB_{+}(0,0) = \PaPB_{+}(0,0) = \{ *_{\col o} \}$. For consistency we will denote $*_{\col c}$ the element of the one-colored unitary operads we will consider ($\CoB_{+}$, $\PaB_{+}$, etc.).

\begin{definition}
  Let $\CoPB_{+}$ be the relative operad over $\CoB_{+}$, defined as a unitary extension of $\CoPB_{+}$. Composition with $*_{\col c} \in \CoB_{+}(0)$ forgets aerial strands, while composition with $*_{\col o}$ forgets terrestrial strands.
\end{definition}

\begin{definition}
  Let $\tMagma_{+}$ be the relative operad over $\magma_{+}$, a unitary extension of $\tMagma$. Composition with nullary elements is given on generators by (it is not necessary to specify $f(*_{\col o})$ as $\tMagma_{+}(0,0)$ is a singleton anyway):
  \[ \mu_{\col{c}}(*_{\col c}, \id_{\col c}) = \mu_{\col{c}}(\id_{\col c}, *_{\col c}) = \id_{\col{c}}, \quad \mu_{\col{o}}(*_{o}, \id_{\col o}) = \mu_{\col{o}}(\id_{\col o}, *_{\col o}) = \id_{\col{o}}. \]

  Let also $\PaPB_{+} = \omega_{+}^{*} \CoPB_{+}$ be the pullback of $\CoPB_{+}$ along $\omega_{+}$, where $\omega_{+}$ is defined as the $\omega$ of \Cref{def.papb} (it is compatible with the unitary extensions).
\end{definition}

\begin{proposition} \label{prop.zigzag-unit} There is a zigzag of categorical equivalences, where $\tMagma_{+}' \subset \SC_{+}$ is the sub--operad generated by $m_{\col{c}}$, $m_{\col{o}}$, $f$ and the nullary elements:
  \[ \pi \SC_{+} \qiso* \left( \pi \SC_{+} \right)_{\mid \tMagma_{+}'} \qiso \PaPB_{+} := \omega_{+}^{*} \CoPB_{+} \qiso \CoPB_{+}, \]
\end{proposition}

\begin{remark}
  In $\tMagma_{+}'$, we have for example:
  \[ \mu_{\col{o}}(\id_{\col o}, *_{\col o}) =
    \begin{gathered} \begin{tikzpicture}
        \draw (-1,0) -- (1,0); \draw (1,0) arc (0:180:1);
        \draw (0,0) arc (0:180:0.5); \draw (-0.5, 0.25) node {1};
      \end{tikzpicture} \end{gathered}
    \neq
    \begin{gathered} \begin{tikzpicture}
        \draw (-1,0) -- (1,0); \draw (1,0) arc (0:180:1);
        \draw (0,0.5) node {1};
      \end{tikzpicture} \end{gathered}
    = \id_{\col{o}},
  \]
  but $\mu_{\col o}(*_{\col o}, *_{\col o}) = *_{\col o}$.  In other words, $*_{\col{o}}$ (and similarly $*_{\col{c}}$ for $\mu_{\col{c}}$) is not a strict unit for $*_{\col{o}}$, but is still idempotent.
\end{remark}

\begin{proof}
  There is an evident morphism $\omega_{+}' : \tMagma_{+}' \to \tMagma_{+}$ sending generators on generators, and we check directly that $\left( \pi \SC_{+} \right)_{\mid \tMagma_{+}'}$ is identified with the pullback $\omega_{+}' \PaPB'$. Since $\omega_{+}$ and $\omega_{+}'$ are both surjective, we obtain the two categorical equivalences $\left( \pi \SC_{+} \right)_{\mid \tMagma_{+}'} \qiso \PaPB_{+} \qiso \CoPB_{+}$.  And since $\tMagma_{+}'$ meets all connected components of $\SC_{+}$ we also have that the inclusion $\left( \pi \SC_{+} \right)_{\mid \tMagma_{+}'} \hookrightarrow \pi \SC_{+}$ is a categorical equivalence.
\end{proof}

The proof of the following proposition is a direct unitary extension of the proof of \Cref{cor.alg-papb} (one also needs to extend the definition of the Drinfeld center to unitary monoidal categories, cf.\ the given references):
\begin{proposition}
  An algebra over $\PaPB_{+}$ is given by:
  \begin{itemize}
  \item a monoidal category $(\cN, \otimes, \mathbbm{1}_{\cN})$ with a strict unit;
  \item a braided monoidal category $(\cM, \otimes, \mathbbm{1}_{\cM}, \tau)$ with a strict unit;
  \item a monoidal functor $F : \cM \to \ZZ(\cN)$ satisfying $F(\mathbbm{1}_{\cM}) = \mathbbm{1}_{\cN}$.
  \end{itemize}

  An algebra over $\CoPB_{+}$ is given by the same data, but where the two tensors products are strictly associative, strictly braided for the second, and the functor is strict braided monoidal.
\end{proposition}

\section{Chord diagrams}
\label{sec.chord-diagrams}

Let $\PP$ an operad in groupoids. Its completion $\widehat{\PP}$ is defined by the Malcev completion arity by arity:
\[ \widehat{\PP}(r) = \mathbb{G} \widehat{\Q[\PP(r)]}, \]
and it is an operad in complete groupoids \cite[§I.9]{Fresse2016a}. Here, $\Q[G]$ is the Hopf groupoid of the groupoid $G$; it has the same objects as $G$, and $\hom_{\Q[G]}(x,y) = \Q[\hom_{G}(x,y)]$ is the free $\Q$-module on the hom-set, equipped with a coalgebra structure where every generator is grouplike. It is completed at the augmentation ideal, and then the functor $\mathbb{G}$ extracts the grouplike elements to define an operad in complete groupoids. This completion is equipped with a canonical completion morphism $\PP \to \widehat{\PP}$.

\begin{definition}
  A morphism of operads in groupoids $\PP \to \QQ$ is called a \textbf{rational categorical equivalence} (denoted $\PP \xrightarrow{\sim_{\Q}} \QQ$) if the induced morphism $\widehat{\PP} \to \widehat{\QQ}$ is a categorical equivalence. We write $\PP \sim_{\Q} \QQ$  if $\PP$ and $\QQ$ can be connected by a zigzag of rational categorical equivalences.
\end{definition}

This definition is motivated by the following remark: if $A$ is an abelian group, then $\widehat{A} = A \otimes_{\mathbb{Z}} \Q$. Examples of rational categorical equivalences include categorical equivalences and the canonical completion morphisms $\PP \to \widehat{\PP}$. We refer to \cite[§I.9]{Fresse2016a} for more details.

\subsection{Drinfeld associators and chord diagrams}
\label{sec.drinf-assoc-chord}

\begin{definition}
  The Drinfeld--Kohno operad $\mathfrak{p}$ is an operad in Lie algebras,\footnote{The monoidal product in the category of Lie algebras is the direct sum.} where in each arity we have the presentation by generators and relations:
  \[ \mathfrak{p}(r) = \operadify{Lie}(t_{ij})_{1 \leq i \neq j \leq r} \bigm/ \bigl( [t_{ij}, t_{kl}], [t_{ik}, t_{ij} + t_{jk}] \bigr), \]
  and operadic composition is given by explicit formulas~\cite[§I.10.2]{Fresse2016a}.
\end{definition}

The universal enveloping algebra functor $\mathbb{U}$ being monoidal, $\mathbb{U}\mathfrak{p}$ is an operad in associative algebras. The algebra $\mathbb{U}\mathfrak{p}(r)$ is generated by chord diagrams with $r$ strands, and composition is given by insertion of a diagram (cf.\ ibid.\ for precise definitions).  We can complete $\mathfrak{p}$ with respect to the weight grading (the weight of $t_{ij}$ is defined to be $1$) to get an operad $\hat{\mathfrak{p}}$ in complete Lie algebras, and we can consider its completed universal enveloping algebra:

\begin{definition}
  The \textbf{operad of completed chord diagrams} $\CDh$ is an operad in groupoids given by $\ob \CDh(r) = *$ and $\Hom_{\CDh(r)}(*,*) = \mathbb{G} \hat{\mathbb{U}} \hat{\mathfrak{p}}(r)$. Operadic composition is induced by the one of $\mathfrak{p}$.
\end{definition}

These operads have unitary extensions: restriction operations forget strands of the chord diagrams, and if a chord was attached to the strand, the diagram is sent to $0$. We thus get unitary operads $\mathfrak{p}_{+}$, $\hat{\mathfrak{p}}_{+}$, and $\CDh_{+}$.

\begin{definition}
  A \textbf{Drinfeld associator} (with parameter $\mu \in \Q^{\times}$) is a morphism $\phi : \PaB_{+} \to \CDh_{+}$ of operads that sends the braiding $\tau \in \PaB_{+}(2)$ to $e^{\mu t_{12}/2} \in \CDh_{+}(2)$. We let $Ass^{\mu}(\Q)$ be the set of such associators.
\end{definition}

If $\phi \in Ass^{\mu}(\Q)$, then the formal series in two variables
\[ \Phi(t_{12}, t_{13}) := \phi(\alpha_{\col{c}}) \in \CDh_{+}(3) \cong \Q \llbracket t_{12}, t_{13} \rrbracket \]
is a Drinfeld associator in the usual sense, satisfying the usual equations (pentagon, hexagon), and vice versa. A Drinfeld associator $\phi$ extends to a categorical equivalence $\phi : \PaBh_{+} \qiso \CDh_{+}$, i.e.\ $\phi$ is a rational equivalence. The set $Ass^{\mu}(\Q)$ is a torsor under the action of the Grothendieck--Teichmüller group $GT^{1}(\Q)$, the group of automorphisms of $\PaBh_{+}$ fixing $\mu_{\col{c}}$ and $\tau$. A theorem of Drinfeld~\cite{Drinfeld1990} states that the set of associators $Ass^{1}(\Q)$ is nonempty.

We can also consider the operad $\PaCDhp$, which is the pullback of $\CDh_{+}$ along the terminal morphism $\magma \to \ob \CDh_{+} = *$. It is used to define the pro-unipotent version of the Grothendieck--Teichmüller group $GRT^{1}(\Q)$, under which $Ass^{\mu}(\Q)$ is a pro-torsor. We recall the following statement~\cite{Fresse2016a}, which is actually a general fact about pullback along morphisms from a free operad: each morphism $\phi : \PaBh_{+} \to \CDh_{+}$ admits a unique lifting
\[\begin{tikzcd}
    {} & \PaCDhp \dar \\
    \PaBh_{+} \urar[dashed]{\tilde\phi_{+}} \rar{\phi} & \CDh_{+}
  \end{tikzcd}\] which is given by the identity on the level of objects. If the morphism $\phi$ came from a Drinfeld associator, then this defines an isomorphism of operads in groupoids:
\[ \tilde\phi_{+} : \PaBh_{+} \xrightarrow{\cong} \PaCDhp. \]

\subsection{Shuffle of operads}
\label{sec.shuffle-operads}

By analogy with the decomposition of \Cref{fig.decomp-papb}, we define a new rational model in groupoids for $\pi \SC_{+}$ that involves the operad of chord diagrams.

\subsubsection{Other description of \texorpdfstring{$\CoPB_+$}{CoPB+}}
\label{sec.other-descr-copb}

\begin{definition}
  Let $\Pi$ be the permutation operad: $\Pi(n) = \Sigma_{n}$, and operadic composition is given by bloc composition of permutations. We also denote $\Pi$ the same operad seen as an operad in discrete groupoids. We also let $\Pi_{+}$ be its obvious unitary extension.
\end{definition}

The following operad is meant to represent the shuffle-type morphisms of \Cref{def.shuffle-type}:

\begin{definition}
  We define the relative (unitary) operad in groupoid $\Sh_{+}$ over $\Pi_{+}$. The set of objects of $\Sh_+(n,m)$ is $\mathrm{Sh}_{n,m} \times \Sigma_{n} \times \Sigma_{m}$, the same as $\CoPB_{+}$ (with the same graphical representation). Operadic composition on the object level is the same as that of $\CoPB_{+}$. On the level of morphisms:
  \[ \Hom_{\Sh_{+}(n,m)}((\mu, \sigma, \sigma'), (\nu, \tau, \tau')) =
    \begin{cases}
      * & \sigma = \tau, \sigma' = \tau', \\
      \varnothing & \text{otherwise},
    \end{cases}
  \]
  and we check that this gives a well-defined relative operad over $\Pi_{+}$ (i.e.\ there are no maps $* \to \varnothing$ to define, and all the maps $\cdot \to *$ are terminal maps).
\end{definition}

Graphically, we simply represent morphisms of $\Sh_{+}$ by an arrow between two bicolored configurations on the interval. Such an arrow exists iff the terrestrial (resp.\ aerial) points of the first configuration are in the same order as the terrestrial (resp.\ aerial) points of the second configuration, so we do not write the labels for the second configuration:
\[ \begin{tikzpicture}[xscale=0.8]
    \draw[|-|] (0,0) -- (6,0);
    \bdot{1,0}{2};
    \bdot{2,0}{1};
    \wdot{3,0}{1};
    \bdot{4,0}{3};
    \wdot{5,0}{2};
    \draw[|-|] (0,-2) -- (6,-2);
    \wdot{1,-2}{};
    \bdot{2,-2}{};
    \bdot{3,-2}{};
    \wdot{4,-2}{};
    \bdot{5,-2}{};
    \draw[->] (3,-0.3) -- (3,-1.7);
\end{tikzpicture} \]

\begin{remark}
  The symmetric groups $\Sigma_{n}$ and $\Sigma_{m}$ act on the left \emph{and} on the right on $\Sh_{+}(n,m)$ (by multiplication on respective factors).
\end{remark}

\begin{lemma}\label{prop.zeta}
  The groupoid $\CoPB_{+}(n,m)$ is isomorphic to $(\CoP_{+}(n) \times \CoB_{+}(m)) \times_{\Sigma_{n} \times \Sigma_{m}} \Sh_{+}(n,m)$.
\end{lemma}
\begin{proof}
  We define:
  \[ \zeta : (\CoP_{+}(n) \times \CoB_{+}(m)) \times_{\Sigma_{n} \times \Sigma_{m}} \Sh_{+}(n,m) \to \CoPB_{+}(n,m) \]
  by a graphical calculus:
  \begin{center}
    \def\svgwidth{\textwidth} 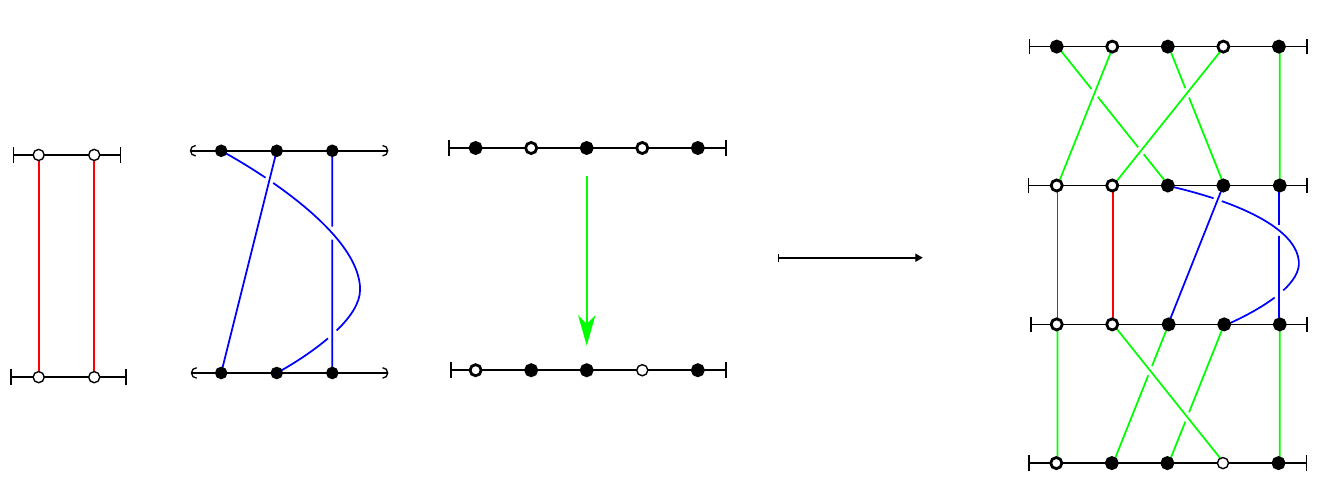
  \end{center}

  Concretely, on objects, we define:
  \begin{align*}
    \zeta
    & : \ob \bigl( \CoP_{+}(n) \times \CoB_{+}(m) \times_{\Sigma_{n} \times \Sigma_{m}} \Sh_{+}(n,m) \bigr) \\
    & \xrightarrow{=} (\Sigma_{n} \times \Sigma_{m}) \times_{\Sigma_{n} \times \Sigma_{m}} \ob \CoPB(n,m) \\
    & \xrightarrow{\cong} \ob \CoPB(n,m).
  \end{align*}

  On morphisms, $\zeta[u,x,\mu]$ (for $u \in \CoP_{+}(n)$, $x \in \CoB_{+}(m)$, $\mu \in \Sh_{+}(n,m)$) is the composition of the unique shuffle type morphism that brings all terrestrial points to the left, then the concatenation of $u$ and $x$, then the unique shuffle-type morphism that brings ground point to their places. We thus get a well-defined (up to isotopy) braid, and it is easy to see that this gives a bijection on morphisms.
\end{proof}

On can thus transport the operadic composition, which will serve as inspiration for \Cref{eq.op-papb-var} to come.

\subsubsection{A variation on \texorpdfstring{$\PaPB_+$}{PaPB+}}
\label{sec.variation-papb_+}

We first define a new operad $\PaPB'_{+}$, a minor variation on $\PaPB_{+}$.

\begin{definition}
  Let $\omega_{+} : \tMagma_{+} \to \ob \CoPB_{+} \cong \ob \Sh_{+}$ be the morphism of \Cref{def.papb}. We define $\PaSh_{+}$ to be the pullback of $\Sh_{+}$ along $\omega$.
\end{definition}

\begin{remark}
  There is a function of sets $U : \ob \PaPB_{+}(0,m) \to \ob \PaB_{+}(m)$ that forgets the second level of parenthesizing.
\end{remark}

\begin{definition}
  Let
  \[ \PaPB'_{+}(n,m) \subset (\PaP_{+}(n) \times \PaB_{+}(m)) \times_{\Sigma_{n} \times \Sigma_{m}} \PaSh_{+}(n,m), \]
  be the full subgroupoid whose objects $[u,x,\mu]$ such that there exists a permutation $\tau \in \Sigma_{m}$ satisfying $U(\mu(*_{\col{o}}, \dots, *_{\col{o}})) = x \cdot \tau$ (this does not depend on the choice of a representative for the coinvariants).
\end{definition}

\begin{example}
  For example,
  \begin{gather*}
    \bigl[
    \begin{tikzpicture}
      \draw[|-|] (0,0) -- (1,0);
      \wdot{0.5,0}{1};
    \end{tikzpicture}
    \times 
    \begin{tikzpicture}[xscale=0.5]
      \draw[|-|] (0,0) -- (4,0);
      \draw[|-|] (2,0) -- (4,0);
      \bdot{1,0}{2}; 
      \draw[|-|] (3,0) -- (4,0);
      \bdot{2.5,0}{1};
      \bdot{3.5,0}{3};
    \end{tikzpicture}
    \times
    \begin{tikzpicture}[xscale=0.75]
      \draw[|-|] (0,0) -- (4,0);
      \draw[|-|] (1,0) -- (2,0);
      \draw[|-] (3,0) -- (4,0);
      \draw[cob] (0.1,0) -- (0.9,0);
      \bdot{0.5,0}{1};
      \wdot{1.5,0}{1};
      \draw[cob] (2.1,0) -- (3.9,0);
      \bdot{2.5,0}{2};
      \bdot{3.5,0}{3};
    \end{tikzpicture}
    \bigr] \in \ob \PaPB'(1,3) \text{, but}
    \\
    \bigl[ *_{\col{o}} \times
    \begin{tikzpicture}[xscale=0.5]
      \draw[|-|] (0,0) -- (4,0);
      \draw[|-|] (2,0) -- (4,0);
      \bdot{0.5,0}{2};
      \draw[|-|] (1,0) -- (2,0);
      \bdot{1.5,0}{1};
      \bdot{3,0}{3};
    \end{tikzpicture}
    \times
    \begin{tikzpicture}[xscale=0.5]
      \draw[|-|] (0,0) -- (4,0);
      \draw[|-|] (2,0) -- (4,0);
      \draw[cob] (0.1,0) -- (1.9,0);
      \bdot{1,0}{1};
      \draw[cob] (2.1,0) -- (3.9,0);
      \draw[|-|] (3,0) -- (4,0);
      \bdot{2.5,0}{2};
      \bdot{3.5,0}{3};
    \end{tikzpicture}
    \bigr] \not\in \ob \PaPB'(0,3).
  \end{gather*}
\end{example}

\begin{lemma}
  The symmetric sequence $\PaPB_{+}'(n)$ is a right module over $\PaB_{+}$, given by:
  \begin{align*}
    \circ_{i}^{\col{c}} : \PaPB_{+}'(n,m) \times \PaB_{+}(k)
    & \to \PaPB_{+}'(n, m+k-1) \\
    [u,x,\mu] \times y
    & \mapsto [u, x \circ_{i} y, \mu \circ_{i}^{\col{c}}
      1_{\Sigma_{k}} ],
  \end{align*}
  where $1_{\Sigma_{k}}$ is seen as a morphism in $\omega^{*}\Pi_{+}(k)$ between the source of $y$ and the target of $y$.
\end{lemma}
\begin{proof}
  The operad $\PaB_{+}$ is a right module over itself, and $\PaSh_{+}(n)$ is a right module over $\magma_{+}$. One can then directly check that the above formula defines a right module structure over $\PaB_{+}$.
\end{proof}

\begin{definition}
  Let $\PP$ be an operad in some symmetric monoidal category. The \textbf{shifted operad} $\PP[\cdot]$ \cite[§10.1]{Fresse2009} is an operad in right modules over $\PP$ (i.e.\ an operad relative over $\PP$), given by $\PP[n](m) = \PP(n+m)$, and the structure maps are induced by the operad structure of $\PP$ (cf.\ ibid.\ for explicit formulas).
\end{definition}

To define the operad structure of $\PaPB'_{+}$, we first define a morphism of colored collections $\rho : \PaPB'_{+} \to \PaB[\cdot]$, that will be similar to the definition of $\zeta$~\ref{prop.zeta}. It is again defined in a graphical way, see \Cref{fig.rho} (the starred numbers correspond to shifted entries). The precise definition involves the inclusion $\iota : \PaP_{+} \hookrightarrow \PaB_{+}$, concatenation, the functor $\PaSh_{+} \to \PaB_{+}$ (picking the unique shuffle-type morphism when it exists), $\eta : \PaP_{+} \to \PaP_{+}[\cdot]$ (where $\eta : \PaP_{+}(n) \cong \PaP_{+}[n](0)$), as well as $U$ on objects -- see \Cref{eq.def-rho}.

\begin{figure}[htbp]
  \centering
  \def\svgwidth{\textwidth} 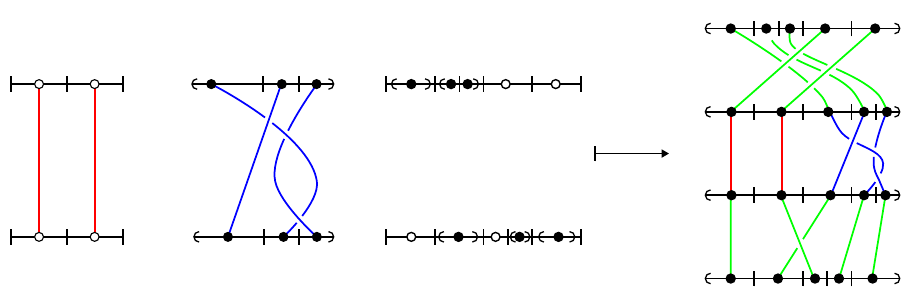
  \caption{Definition of $\rho : \PaPB'_{+}(2,3) \to \PaB_{+}[2](3)$}
  \label{fig.rho}
\end{figure}

The operad structure is then defined by (where $\sigma \in \Sigma_{m}$ is such that $U(\mu(*_{\col{o}}, \dots, *_{\col{o}})) = x \cdot \sigma$, and where $y_{i}$ is seen as an element of $\PaPB'_{+}[0](l_{i})$):
\begin{align}
  \gamma : \PaPB_+'(r,s) \times \PaPB_+'(k_{1}, l_{1}) \times \dots \times \PaPB_+'(k_{r}, l_{r})
  & \to \PaPB_+'(\sum k_{i}, s + \sum l_{i})
    \nonumber \\
  [u,x,\mu] \times [v_{1}, y_{1}, \sigma_{1}] \times \dots \times
  [v_{r}, y_{r}, \sigma_{r}]
  & \nonumber \\
  \mapsto \bigl[ u(v_{1}, \dots, v_{r}), \sigma^{-1} \cdot
  \underbrace{\rho[u,x,\mu]}_{\PaPB'_{+}[r](s)} (y_{1},
  \dots, y_{r})
  &, \mu(\sigma_{1}, \dots, \sigma_{r}) \bigr]
    \label{eq.op-papb-var}
\end{align}
and where the identity of the operad is $\id = \bigl[ \tikz{\draw[|-|] (0,0) -- (1,0); \wdot{0.5,0}{1};} \times *_{\col{c}} \times \tikz{\draw[|-|] (0,0) -- (1,0); \wdot{0.5,0}{1};} \bigr] \in \PaPB_{+}'(1,0)$.

\begin{proposition} \label{prop.papb-var-operade}
  Given this operadic composition, this identity and this right $\PaB_{+}$-module structure, $\PaPB_{+}'$ is an operad relative over $\PaB_{+}$.
\end{proposition}
\begin{proof}
  The $\sigma^{-1}$ in the formula ensures that this $\gamma$ is well-defined (it does not depend on the representative in the coinvariants). The fact that $\gamma$ is equivariant is a direct consequence of the fact that the operad structures of $\PaP_{+}$, $\PaSh_{+}$ and $\PaB_{+}[\cdot]$ are equivariant.

  Let $[u,x,\mu] \in \PaPB'_{+}(n,m)$. The identity $\id([u,x,\mu]) = [u,x,\mu]$ is immediate by definition, and from the condition on the objects of $\PaPB'_{+}$, one can also show the identity $[u,x,\mu](\id, \dots, \id) = [u,x,\mu]$.

  To see that $\gamma$ is a morphism of right $\PaB_{+}$-modules, it is enough to have the identity $\rho[u,x,\mu](y_{1}, \dots, y_{j} \circ_{i} z, \dots, y_{r}) = \rho[u,x,\mu](y_{1}, \dots, y_{r}) \circ_{l_{1} + \dots + l_{j-1} + i} z$ for $z \in \PaB_{+}(m)$; but since $\PaB_{+}[\cdot]$ is an operad, this identity is satisfied.

  Finally, associativity of $\gamma$ follows from the condition on objects (to show that $\ob \gamma$ is associative), and from the fact that $\PaB_{+}[\cdot]$ is an operad (to show associativity on morphisms).
\end{proof}

\begin{proposition} \label{prop.papb-prime}
  There exists a categorical equivalence $\PaPB_{+}' \qiso \CoPB_{+}$.
\end{proposition}
\begin{proof}
  This equivalence is given in arity $(n,m)$ by the restriction to $\PaPB'_{+}(n,m)$ of the composite:
  \begin{align*}
    & (\PaP_{+}(n) \times \PaB_{+}(m)) \times_{\Sigma_{n} \times
    \Sigma_{m}} \PaSh_{+}(n,m)
    \\
    & \to (\CoP_{+}(n) \times \CoB_{+}(m)) \times_{\Sigma_{n} \times \Sigma_{m}} \Sh_{+}(n,m)
    \\
    & \xrightarrow[\zeta]{\cong} \CoPB_{+}(n,m)
  \end{align*}

  By construction (the operad structure of $\PaPB'_{+}$ is directly mimicked from the operad structure of $\CoPB_{+}$), this yields a morphism of operads $\PaPB'_{+} \to \CoPB_{+}$. Since $\PaP_{+} \to \CoP_{+}$, $\PaB_{+} \to \CoB_{+}$ and $\PaSh_{+} \to \Sh_{+}$ are categorical, their product is too, thus the above morphism yields a categorical equivalence.
\end{proof}

\subsubsection{An operad defined from chord diagrams}
\label{sec.an-operad-defined}

We choose a Drinfeld associator $\phi : \PaB_{+} \to \CDh_{+}$; let $\tilde\phi_{+} : \PaB_{+} \to \PaCDhp$ be its unique lifting. Similarly to the definition of $\PaPB'_{+}$, we will define a relative operad $\PaPCD^{\phi}$ over $\PaCDhp$, that will combine parenthesized shuffles, parenthesized permutations and parenthesized chords diagrams.

\begin{definition}
  Let
  \[ \PaPCD^{\phi}(n,m) \subset (\PaP_{+}(n) \times \PaCDhp(m)) \times_{\Sigma_{n} \times \Sigma_{m}} \PaSh_{+}(n,m) \]
  be the full subgroupoid whose objects are classes $[u,\alpha,\mu]$ such that there exists $\sigma \in \Sigma_{m}$ satisfying $U(\mu(*_{\col{o}}, \dots, *_{\col{o}})) = \alpha \cdot \sigma$.
\end{definition}

\begin{remark}
  The objects of $\PaPCD^{\phi}$ are the same as the objects of $\PaPB'_{+}$.
\end{remark}

We also define a morphism $\rho_{\phi} : \PaPCD^{\phi} \to \PaCDhp[\cdot]$. Its definition is similar to that of \Cref{fig.rho}, but one cannot directly use graphical calculus anymore. In the picture, concatenation in $\PaB_{+}$ and $\PaB_{+}[\cdot]$ corresponded to $m_{\col{c}}$.  Pre- and post-composition by shuffles in $\PaPB'_{+}$ came from a morphism of operads $\sigma : \PaSh_{+} \to \PaB_{+}[\cdot]$. We have $\ob\sigma = \id$, and $\sigma(\mu \to \mu')$, denoted $\sigma^{\mu}_{\mu'}$ to simplify, is the unique morphism $\PaB[\cdot]$ of shuffle-type between the corresponding objects.

We also recall the canonical morphism $\eta : \PP \to \PP[\cdot]$, given in arity $m$ by $\PP(m) \cong \PP[m](0)$. Finally, we could define $\rho$ by the following formula (where $x \in \PaB_{+}(m)$ is identified with $x \in \PaB_{+}[0](m)$):
\begin{equation}
  \label{eq.def-rho}
  \rho[u,x,\mu] =
  \sigma^{m_{\col{c}}(\iota(\eta(\operatorname{tgt}(u))),
    \operatorname{tgt}(x))}_{\operatorname{tgt}(\mu)} \circ
  m_{\col{c}}(\iota(\eta(u)), x)
  \circ \sigma_{m_{\col{c}}(\iota(\eta(\operatorname{src}(u))),
    \operatorname{src}(x))}^{\operatorname{src}(\mu)}.
\end{equation}

By analogy, we define:
\[\rho_{\phi} : \PaPCD^{\phi}(n,m) \to \PaCDhp[n](m) \]

To simplify, we let $\tilde\phi_{+}(m_{\col{c}}) = \tilde{m}_{\col{c}}$, $\tilde\phi_{+} \circ \sigma = \tilde\sigma$, $\tilde\phi_{+}\iota\eta = \tilde\iota$, and we again identify $\alpha \in \PaCDhp(m)$ with $\alpha \in \PaCDhp[0](m)$. Then $\rho_{\phi}$ is given by:
\[ [u,\alpha,\mu] \mapsto \tilde\sigma^{\tilde{m}_{\col{c}}(\tilde\iota(\operatorname{tgt}(u)), \operatorname{tgt}(\alpha))}_{\operatorname{tgt}(\mu)} \circ \tilde{m}_{\col{c}}(\tilde\iota(u), \alpha) \circ \tilde\phi_{+} \bigl( \tilde\sigma_{\tilde{m}_{\col{c}}(\tilde\iota(\operatorname{src}(u)), \operatorname{src}(\alpha))}^{\operatorname{src}(\mu)} \bigr). \]

Graphically, $\rho_{\phi}$ looks like \Cref{fig.rho-phi}, where the gray boxes represent applications of the associator. We then define an operadic composition in a similar manner to \Cref{eq.op-papb-var}, replacing $\rho$ by $\rho_{\phi}$. We also define a right $\PaCDhp$-module similar to that of $\PaPB'_{+}$.

\begin{figure}[htb]
  \centering
  \def\svgwidth{0.35\textwidth} 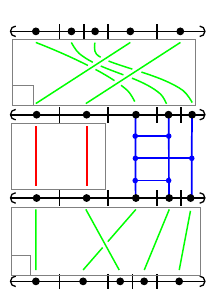
  \caption{Graphical representation of $\rho_\phi$}
  \label{fig.rho-phi}
\end{figure}

\begin{theorem} \label{thm.rat-model}
  The data $\PaPCD^{\phi}$, equipped with these structures, is a relative operad over $\PaCDhp$, and the morphism $\PaPB'_{+} \to \PaPCD^{\phi}$ induced by $\tilde\phi_{+}$ is a rational categorical equivalence of operads. There is thus a zigzag:
  \[ \pi \SC_{+} \qiso* \left( \pi \SC_{+} \right)_{\tMagma'_{+}} \qiso \PaPB_{+} \qiso \CoPB_{+} \qiso* \PaPB'_{+} \xrightarrow{\sim_{\Q}} \PaPCD^{\phi}. \]
\end{theorem}

\begin{proof}
  The proof that $\PaPCD^{\phi}$ is a relative operad is identical to the proof of \Cref{prop.papb-var-operade}, and the fact that the morphism induced by $\tilde\phi_{+}$ is a morphism of operads follows by a direct inspection of the definitions.

  The morphism $\tilde\phi_{+} : \PaB_{+} \to \PaCDhp$ is not a
  categorical equivalence, but it factors as:
  \[\begin{tikzcd}
      \PaB_{+} \ar{rr}{\tilde\phi_{+}} \ar[swap]{dr}{\sim_{\Q}} && \PaCDhp \\
      {} & \PaBh_{+} \ar[dashed,swap]{ur}{\cong}
    \end{tikzcd}\]
  where the dashed morphism $\PaBh_{+} \to \PaCDhp$ is an isomorphism of operads in groupoids. It follows that $\tilde\phi_{+}$ is a rational equivalence of operads in groupoids, thus $\PaPB'_{+} \to \PaPCD^{\phi}$ is also a rational categorical equivalence. By combining this fact with \Cref{prop.zigzag-unit,prop.papb-prime}, we finally get the zigzag of the theorem.
\end{proof}

\subsection{Non-formality}
\label{sec.non-formality}

A theorem of Livernet~\cite[Theorem 3.1]{Livernet2015} states that the Swiss-Cheese operad is not \emph{formal}: its homology $H_{*}(\SC)$ is not equivalent to its operad of chains $C_{*}(\SC)$. We give an interpretation of this fact here.

We consider a stronger version of formality, which involves the models of rational homotopy theory of Sullivan (see~\cite[§II]{Fresse2016a} for the applications to operads). Let
\[ \rlz{-} : \cdga_{+}^{\mathrm{op}} \to s\Set \]
be the \emph{derived} Sullivan realization functor,\footnote{The underived realization functor maps a commutative, unitary differential graded algebra $A$ to the simplicial set $\langle A \rangle = \hom_{\cdga_{+}}(A, \Omega_{PL}^{*}(\Delta^{\bullet}))$.  The derived version takes a cofibrant replacement first.} that uses commutative dg-algebras as rational models for spaces.

We rely on \emph{co}homological models to study rational homotopy theory, so we consider dual structures of our objects and we use cooperads rather than operads. For example, the cooperad $\Com^{*}$ governing cocommutative coalgebras is dual to the operad $\Com$ governing commutative algebras.

To encode the rational homotopy-theoretic information, we add commutative structures to our objects, and we consider \textbf{Hopf cooperads}, i.e.\ cooperads in the category of commutative algebras. The Sullivan realization of a Hopf cooperad is a simplicial operad.

\subsubsection{Splitting of $H_{*}(\SC)$ as a Voronov product}
\label{sec.homology-dd_n-sc}

A theorem of Cohen~\cite{Cohen1976} describes the homology of the little disks operads $\operadify{e}_{n} := H_{*}(\DD_{n})$. In low dimensions, $\operadify{e}_{1} \cong \Ass$ is the operad governing associative algebras, while $\operadify{e}_{2} \cong \Ger$ is the operad governing Gerstenhaber algebras. These are \textbf{Hopf operads} (i.e.\ operads in the category of cocommutative coalgebras): the coproduct of the product of either $\Ass$ or $\Ger$ is $\Delta(\mu) = \mu \otimes \mu$, while the coproduct of the bracket of $\Ger$ is $\Delta(\lambda) = \mu \otimes \lambda + \lambda \otimes \mu$. Their duals $\Ass^{*}$ and $\Ger^{*}$ are Hopf cooperads.

The homology of the Swiss-Cheese operad $\Sc := H_{*}(\SC)$ governs the action of a Gerstenhaber algebra on an associative algebra. A theorem of Voronov~\cite[Theorem 3.3]{Voronov1999} (see also~\cite[Theorem 6.1.1]{HoefelLivernet2012} for this particular variant) states that an algebra over $H_{*}(\SC)$ is a triple $(B,A,f)$ where $B$ is a Gerstenhaber algebra, $A$ is an associative algebra, and $f : B \to A$ is a central morphism of associative algebras, which thus makes $A$ into an associative algebra over the commutative algebra $B$. Let us note that \Cref{cor.alg-papb} is a categorical analogue of Voronov's theorem, the Drinfeld center of a monoidal category replacing the center of an associative algebra.

This theorem can be interpreted in the following way.

\begin{definition} \label{def.voronov-product}
  Given two operads $\PP$ and $\QQ$ and a morphism $\Com \to \PP$, one can define the \textbf{Voronov product} $\PP \otimes \QQ$~\cite{Voronov1999}. It is a relative operad over $\PP$, defined by $(\PP \otimes \QQ)(n,m) = \PP(m) \otimes \QQ(n)$. Insertion of a closed-output operation:
  \[ \circ_{i}^{\col c} : \bigl( \PP(m) \otimes \QQ(n) \bigr) \otimes \PP(m') \to \PP(m+m'-1) \otimes \QQ(n) \]
  uses the operad structure of $\PP$, while insertion of an open-output operation:
  \[ \circ_{j}^{\col o} : \bigl( \PP(m) \otimes \QQ(n) \bigr) \otimes \bigl( \PP(m') \otimes \QQ(n') \bigr) \to \PP(m+m') \otimes \QQ(n+n'-1) \]
  uses the operad structure of $\QQ$ and the commutative product $\Com \to \PP$.
\end{definition}

Algebras over $\PP \otimes \QQ$ are triplets $(B,A,\nu)$ where $B$ is a $\PP$-algebra, $A$ is a $\QQ$-algebra, and $\nu : B \otimes A \to A$ is an action that makes $A$ into a $\QQ$-algebra over the commutative algebra $B$ (cf.\ ibid.\ for the definition).

\begin{remark}
  An Eckmann--Hilton-type argument shows that the algebra structure of $\PP$ defined by the morphism $\Com \to \PP$ has to be commutative for the composition product to even be associative.
\end{remark}

Voronov's version of the Swiss-Cheese $\SC^{\mathrm{vor}}$ operad then
satisfies:
\[ \Sc^{\mathrm{vor}} := H_{*}(\SC^{\mathrm{vor}}) \cong \Ger \otimes \Ass. \]
This isomorphism is moreover an isomorphism of Hopf operads.

In the case of $\Sc = H_{*}(\SC)$, one has to use the unital structures of $\Ass$ and $\Ger$. We have (and this is still an isomorphism of Hopf operads):
\[ \Sc_{+} := H_{*}(\SC_{+}) \cong \Ger_{+} \otimes \Ass_{+}. \]

If we remove the components with zero closed inputs and zero open inputs, we then get $\Sc = H_{*}(\SC)$, a relative operad over $\Ger$ whose algebras are described above \Cref{def.voronov-product}. But one should not forget that:
\begin{align*}
  \Sc(0,m) & = \Ger_{+}(m) \otimes \Ass_{+}(0) = \Ger(m) \\
  \neq \Sc^{\mathrm{vor}}(0,m) & = \Ger(m) \otimes \Ass(0) = 0.
\end{align*}

Indeed, we still keep the components that have a nonzero number of total inputs. We have in particular that $\Sc(0,1) \cong \Ger(1) = \Q$ is spanned by the morphism between the Gerstenhaber algebra to the associative algebra. We use the notation:
\[ \Sc = \Ger_{+} \otimes_{0} \Ass_{+} \]
to express the fact that $\Sc$ is obtained as the Voronov product $\Ger_{+} \otimes \Ass_{+}$ from which we remove the components with zero closed inputs and zero open inputs.

\subsubsection{Comparison}
\label{sec.comparison}

By theorems of Kontsevich~\cite{Kontsevich1999} ($\K = \mathbb{R}$, $n \ge 2$) and Tamarkin~\cite{Tamarkin2003} ($\K = \Q$, $n = 2$), the little disks operads are formal: $C_{*}(\DD_{n}^{+}) \simeq H_{*}(\DD_{n}^{+})$. Fresse and Willwacher~\cite{FresseWillwacher2015} give another proof of this result ($\K = \Q$, $n \ge 3$), and show that it can be enhanced in the rational homotopy context: there is a rational equivalence of simplicial operads $\DD_{n} \simeq_{\Q} \rlz{H^{*}(\DD_{n})}$, which implies the rational formality of $\DD_{n}$. In low dimensions, we thus have $\DD_{1} \simeq_{\Q} \rlz{\Ass^{*}}$ (easy computation). Tamarkin proves that, from the existence of rational Drinfeld associators, it follows that $\DD_{2} \simeq_{\Q} \rlz{\Ger^{*}}$. These rational equivalences are also compatible with the unital structures.

Given two (Hopf) cooperads $\PP_{c}$, $\QQ_{c}$, and a morphism $\PP_{c} \to \Com^{*}$, one can define the Voronov product $\PP_{c} \otimes \QQ_{c}$, similarly to \Cref{def.voronov-product}. This is a relative (Hopf) cooperad under $\PP_{c}$. It is defined by formulas that are formally dual to the ones defining $\PP \otimes \QQ$. If these cooperads admit counital extensions, we can similarly define $\PP^{+}_{c} \otimes_0 \QQ^{+}_{c}$.

The inclusion $\DD_{1}^{+} \hookrightarrow \DD_{2}^{+}$ induces in cohomology a morphism of Hopf cooperads $\Ger_{+}^{*} \to \Ass_{+}^{*}$. This morphism factors through a morphism $\Ger_{+}^{*} \to \Com_{+}^{*}$ (the coproduct of a Gerstenhaber coalgebra is cocommutative). We then get an isomorphism:
\[ \Sc^{*} = H^{*}(\SC) \cong (\Ger_{+} \otimes_0 \Ass_{+})^{*} \cong \Ger_{+}^{*} \otimes_0 \Ass_{+}^{*}. \]

The morphism $\Ger_{+}^{*} \to \Com_{+}^{*}$ induces a morphism of simplicial operads $\Com_{+} \cong \rlz{\Com_{+}^{*}} \to \rlz{\Ger_{+}^{*}}$. Since the realization functor is monoidal, we get:\footnote{The monoidal structure on simplicial sets being the Cartesian product, we denote the Voronov product of two simplicial operads with $\times$ instead of $\otimes$, and we use $\times_{0}$ to say that we remove the components with zero inputs.}
\[ \rlz{H^{*}(\SC)} \simeq \rlz{\Ger_{+}^{*}} \times_{0} \rlz{\Ass_{+}^{*}}. \]

It is known that $\PaP \simeq \pi \DD_{1}$ and $\CDh \simeq_{\Q} \pi \DD_{2}$~\cite[§5,§10]{Fresse2016a}. There is an obvious morphism $\Com \to \CDh$ sending the generator to the empty chord diagram (see~\cite{FresseWillwacher2015}), which we can use to build the Voronov product of the operads in groupoids $\CDh$ and $\PaP$. Since the fundamental groupoid functor is monoidal too, we finally have:
\[ \pi\rlz{\Sc^{*}} \simeq \pi\rlz{\Ger_{+}^{*}} \times \pi\rlz{\Ass_{+}^{*}} \simeq_{\Q} \CDh_{+} \times_{0} \PaP_{+}. \]

The operad $\SC$ is not formal~\cite{Livernet2015}, therefore $\pi \SC$ is not equivalent to $\pi\rlz{\Sc^{*}} \simeq_{\Q} \PaP_{+} \mathop{\bar\times} \CDh_{+}$. Thus, our construction $\PaPCD^{\phi}$ rectifies the model arising from the homology of $\SC$ to retrieve, from $\PaP$ and $\CDh$, an operad in groupoids which is actually rationally equivalent to $\pi \SC$.

\bibliographystyle{alpha}
\bibliography{article}
\end{document}

%% file: fig/sc.pdf_tex
%% Creator: Inkscape 0.91_64bit, www.inkscape.org
%% Accompanies image file 'sc.pdf' (pdf, eps, ps)
\begingroup%
  \makeatletter%
  \providecommand\color[2][]{%
    \errmessage{(Inkscape) Color is used for the text in Inkscape, but the package 'color.sty' is not loaded}%
    \renewcommand\color[2][]{}%
  }%
  \providecommand\transparent[1]{%
    \errmessage{(Inkscape) Transparency is used (non-zero) for the text in Inkscape, but the package 'transparent.sty' is not loaded}%
    \renewcommand\transparent[1]{}%
  }%
  \providecommand\rotatebox[2]{#2}%
  \ifx\svgwidth\undefined%
    \setlength{\unitlength}{134.169338bp}%
    \ifx\svgscale\undefined%
      \relax%
    \else%
      \setlength{\unitlength}{\unitlength * \real{\svgscale}}%
    \fi%
  \else%
    \setlength{\unitlength}{\svgwidth}%
  \fi%
  \global\let\svgwidth\undefined%
  \global\let\svgscale\undefined%
  \makeatother%
  \begin{picture}(1,0.523016)%
    \put(0,0){\includegraphics[width=\unitlength,page=1]{sc.pdf}}%
    \put(0.231682,0.05283){\color[rgb]{0,0,0}\makebox(0,0)[lb]{\smash{1}}}%
    \put(0.523424,0.046442){\color[rgb]{0,0,0}\makebox(0,0)[lb]{\smash{2}}}%
    \put(0.793872,0.043247){\color[rgb]{0,0,0}\makebox(0,0)[lb]{\smash{3}}}%
    \put(0.308345,0.242356){\color[rgb]{0,0,0}\makebox(0,0)[lb]{\smash{1}}}%
    \put(0.635148,0.334366){\color[rgb]{0,0,0}\makebox(0,0)[lb]{\smash{2}}}%
  \end{picture}%
\endgroup%

%% file: fig/example_copb.pdf_tex
%% Creator: Inkscape 0.91_64bit, www.inkscape.org
%% Accompanies image file 'example_copb.pdf' (pdf, eps, ps)
\begingroup%
  \makeatletter%
  \providecommand\color[2][]{%
    \errmessage{(Inkscape) Color is used for the text in Inkscape, but the package 'color.sty' is not loaded}%
    \renewcommand\color[2][]{}%
  }%
  \providecommand\transparent[1]{%
    \errmessage{(Inkscape) Transparency is used (non-zero) for the text in Inkscape, but the package 'transparent.sty' is not loaded}%
    \renewcommand\transparent[1]{}%
  }%
  \providecommand\rotatebox[2]{#2}%
  \ifx\svgwidth\undefined%
    \setlength{\unitlength}{142.710915bp}%
    \ifx\svgscale\undefined%
      \relax%
    \else%
      \setlength{\unitlength}{\unitlength * \real{\svgscale}}%
    \fi%
  \else%
    \setlength{\unitlength}{\svgwidth}%
  \fi%
  \global\let\svgwidth\undefined%
  \global\let\svgscale\undefined%
  \makeatother%
  \begin{picture}(1,0.579218)%
    \put(0,0){\includegraphics[width=\unitlength,page=1]{example_copb.pdf}}%
    \put(0.106752,0.521689){\color[rgb]{0,0,0}\makebox(0,0)[lb]{\smash{$1$}}}%
    \put(0.695354,0.521689){\color[rgb]{0,0,0}\makebox(0,0)[lb]{\smash{$2$}}}%
    \put(0.302953,0.521689){\color[rgb]{0,0,0}\makebox(0,0)[lb]{\smash{$3$}}}%
    \put(0.499153,0.521689){\color[rgb]{0,0,0}\makebox(0,0)[lb]{\smash{$1$}}}%
    \put(0.863526,0.521689){\color[rgb]{0,0,0}\makebox(0,0)[lb]{\smash{$2$}}}%
  \end{picture}%
\endgroup%

%% file: fig/comp_copb_pid2.pdf_tex
%% Creator: Inkscape 0.91_64bit, www.inkscape.org
%% Accompanies image file 'comp_copb_pid2.pdf' (pdf, eps, ps)
\begingroup%
  \makeatletter%
  \providecommand\color[2][]{%
    \errmessage{(Inkscape) Color is used for the text in Inkscape, but the package 'color.sty' is not loaded}%
    \renewcommand\color[2][]{}%
  }%
  \providecommand\transparent[1]{%
    \errmessage{(Inkscape) Transparency is used (non-zero) for the text in Inkscape, but the package 'transparent.sty' is not loaded}%
    \renewcommand\transparent[1]{}%
  }%
  \providecommand\rotatebox[2]{#2}%
  \ifx\svgwidth\undefined%
    \setlength{\unitlength}{70.069291bp}%
    \ifx\svgscale\undefined%
      \relax%
    \else%
      \setlength{\unitlength}{\unitlength * \real{\svgscale}}%
    \fi%
  \else%
    \setlength{\unitlength}{\svgwidth}%
  \fi%
  \global\let\svgwidth\undefined%
  \global\let\svgscale\undefined%
  \makeatother%
  \begin{picture}(1,1.284478)%
    \put(0,0){\includegraphics[width=\unitlength,page=1]{comp_copb_pid2.pdf}}%
  \end{picture}%
\endgroup%

%% file: fig/exa_rmod_1.pdf_tex
%% Creator: Inkscape 0.91_64bit, www.inkscape.org
%% Accompanies image file 'exa_rmod_1.pdf' (pdf, eps, ps)
\begingroup%
  \makeatletter%
  \providecommand\color[2][]{%
    \errmessage{(Inkscape) Color is used for the text in Inkscape, but the package 'color.sty' is not loaded}%
    \renewcommand\color[2][]{}%
  }%
  \providecommand\transparent[1]{%
    \errmessage{(Inkscape) Transparency is used (non-zero) for the text in Inkscape, but the package 'transparent.sty' is not loaded}%
    \renewcommand\transparent[1]{}%
  }%
  \providecommand\rotatebox[2]{#2}%
  \ifx\svgwidth\undefined%
    \setlength{\unitlength}{43.02718bp}%
    \ifx\svgscale\undefined%
      \relax%
    \else%
      \setlength{\unitlength}{\unitlength * \real{\svgscale}}%
    \fi%
  \else%
    \setlength{\unitlength}{\svgwidth}%
  \fi%
  \global\let\svgwidth\undefined%
  \global\let\svgscale\undefined%
  \makeatother%
  \begin{picture}(1,1.780504)%
    \put(0,0){\includegraphics[width=\unitlength,page=1]{exa_rmod_1.pdf}}%
    \put(0.244367,1.5711){\color[rgb]{0,0,0}\makebox(0,0)[lb]{\smash{}}}%
    \put(0.225774,1.589693){\color[rgb]{0,0,0}\makebox(0,0)[lb]{\smash{$1$}}}%
    \put(0.672003,1.589693){\color[rgb]{0,0,0}\makebox(0,0)[lb]{\smash{$2$}}}%
    \put(0,0){\includegraphics[width=\unitlength,page=2]{exa_rmod_1.pdf}}%
  \end{picture}%
\endgroup%

%% file: fig/exa_rmod_2.pdf_tex
%% Creator: Inkscape 0.91_64bit, www.inkscape.org
%% Accompanies image file 'exa_rmod_2.pdf' (pdf, eps, ps)
\begingroup%
  \makeatletter%
  \providecommand\color[2][]{%
    \errmessage{(Inkscape) Color is used for the text in Inkscape, but the package 'color.sty' is not loaded}%
    \renewcommand\color[2][]{}%
  }%
  \providecommand\transparent[1]{%
    \errmessage{(Inkscape) Transparency is used (non-zero) for the text in Inkscape, but the package 'transparent.sty' is not loaded}%
    \renewcommand\transparent[1]{}%
  }%
  \providecommand\rotatebox[2]{#2}%
  \ifx\svgwidth\undefined%
    \setlength{\unitlength}{142.710915bp}%
    \ifx\svgscale\undefined%
      \relax%
    \else%
      \setlength{\unitlength}{\unitlength * \real{\svgscale}}%
    \fi%
  \else%
    \setlength{\unitlength}{\svgwidth}%
  \fi%
  \global\let\svgwidth\undefined%
  \global\let\svgscale\undefined%
  \makeatother%
  \begin{picture}(1,0.579218)%
    \put(0,0){\includegraphics[width=\unitlength,page=1]{exa_rmod_2.pdf}}%
    \put(0.106752,0.521689){\color[rgb]{0,0,0}\makebox(0,0)[lb]{\smash{$1$}}}%
    \put(0.695354,0.521689){\color[rgb]{0,0,0}\makebox(0,0)[lb]{\smash{$2$}}}%
    \put(0.443096,0.521689){\color[rgb]{0,0,0}\makebox(0,0)[lb]{\smash{$1$}}}%
    \put(0.527182,0.521689){\color[rgb]{0,0,0}\makebox(0,0)[lb]{\smash{$2$}}}%
    \put(0.863526,0.521689){\color[rgb]{0,0,0}\makebox(0,0)[lb]{\smash{$3$}}}%
    \put(0.302953,0.521689){\color[rgb]{0,0,0}\makebox(0,0)[lb]{\smash{$4$}}}%
  \end{picture}%
\endgroup%

%% file: fig/exa_op_1.pdf_tex
%% Creator: Inkscape 0.91_64bit, www.inkscape.org
%% Accompanies image file 'exa_op_1.pdf' (pdf, eps, ps)
\begingroup%
  \makeatletter%
  \providecommand\color[2][]{%
    \errmessage{(Inkscape) Color is used for the text in Inkscape, but the package 'color.sty' is not loaded}%
    \renewcommand\color[2][]{}%
  }%
  \providecommand\transparent[1]{%
    \errmessage{(Inkscape) Transparency is used (non-zero) for the text in Inkscape, but the package 'transparent.sty' is not loaded}%
    \renewcommand\transparent[1]{}%
  }%
  \providecommand\rotatebox[2]{#2}%
  \ifx\svgwidth\undefined%
    \setlength{\unitlength}{49.347416bp}%
    \ifx\svgscale\undefined%
      \relax%
    \else%
      \setlength{\unitlength}{\unitlength * \real{\svgscale}}%
    \fi%
  \else%
    \setlength{\unitlength}{\svgwidth}%
  \fi%
  \global\let\svgwidth\undefined%
  \global\let\svgscale\undefined%
  \makeatother%
  \begin{picture}(1,1.675076)%
    \put(0,0){\includegraphics[width=\unitlength,page=1]{exa_op_1.pdf}}%
    \put(0.227664,1.508704){\color[rgb]{0,0,0}\makebox(0,0)[lb]{\smash{$1$}}}%
    \put(0.714012,1.508704){\color[rgb]{0,0,0}\makebox(0,0)[lb]{\smash{$1$}}}%
  \end{picture}%
\endgroup%

%% file: fig/exa_op_2.pdf_tex
%% Creator: Inkscape 0.91_64bit, www.inkscape.org
%% Accompanies image file 'exa_op_2.pdf' (pdf, eps, ps)
\begingroup%
  \makeatletter%
  \providecommand\color[2][]{%
    \errmessage{(Inkscape) Color is used for the text in Inkscape, but the package 'color.sty' is not loaded}%
    \renewcommand\color[2][]{}%
  }%
  \providecommand\transparent[1]{%
    \errmessage{(Inkscape) Transparency is used (non-zero) for the text in Inkscape, but the package 'transparent.sty' is not loaded}%
    \renewcommand\transparent[1]{}%
  }%
  \providecommand\rotatebox[2]{#2}%
  \ifx\svgwidth\undefined%
    \setlength{\unitlength}{142.710915bp}%
    \ifx\svgscale\undefined%
      \relax%
    \else%
      \setlength{\unitlength}{\unitlength * \real{\svgscale}}%
    \fi%
  \else%
    \setlength{\unitlength}{\svgwidth}%
  \fi%
  \global\let\svgwidth\undefined%
  \global\let\svgscale\undefined%
  \makeatother%
  \begin{picture}(1,0.579218)%
    \put(0,0){\includegraphics[width=\unitlength,page=1]{exa_op_2.pdf}}%
    \put(0.106752,0.521689){\color[rgb]{0,0,0}\makebox(0,0)[lb]{\smash{$1$}}}%
    \put(0.695354,0.521689){\color[rgb]{0,0,0}\makebox(0,0)[lb]{\smash{$2$}}}%
    \put(0.302953,0.521689){\color[rgb]{0,0,0}\makebox(0,0)[lb]{\smash{$3$}}}%
    \put(0.499153,0.521689){\color[rgb]{0,0,0}\makebox(0,0)[lb]{\smash{$1$}}}%
    \put(0.863526,0.521689){\color[rgb]{0,0,0}\makebox(0,0)[lb]{\smash{$2$}}}%
    \put(0.190838,0.521689){\color[rgb]{0,0,0}\makebox(0,0)[lb]{\smash{$4$}}}%
  \end{picture}%
\endgroup%

%% file: fig/free_operad.pdf_tex
%% Creator: Inkscape 0.91_64bit, www.inkscape.org
%% Accompanies image file 'free_operad.pdf' (pdf, eps, ps)
\begingroup%
  \makeatletter%
  \providecommand\color[2][]{%
    \errmessage{(Inkscape) Color is used for the text in Inkscape, but the package 'color.sty' is not loaded}%
    \renewcommand\color[2][]{}%
  }%
  \providecommand\transparent[1]{%
    \errmessage{(Inkscape) Transparency is used (non-zero) for the text in Inkscape, but the package 'transparent.sty' is not loaded}%
    \renewcommand\transparent[1]{}%
  }%
  \providecommand\rotatebox[2]{#2}%
  \ifx\svgwidth\undefined%
    \setlength{\unitlength}{133.925291bp}%
    \ifx\svgscale\undefined%
      \relax%
    \else%
      \setlength{\unitlength}{\unitlength * \real{\svgscale}}%
    \fi%
  \else%
    \setlength{\unitlength}{\svgwidth}%
  \fi%
  \global\let\svgwidth\undefined%
  \global\let\svgscale\undefined%
  \makeatother%
  \begin{picture}(1,0.527139)%
    \put(0,0){\includegraphics[width=\unitlength,page=1]{free_operad.pdf}}%
    \put(0.193859,0.161543){\color[rgb]{0,0,0}\makebox(0,0)[lb]{\smash{1}}}%
    \put(0.313329,0.161543){\color[rgb]{0,0,0}\makebox(0,0)[lb]{\smash{2}}}%
    \put(0.612003,0.101808){\color[rgb]{0,0,0}\makebox(0,0)[lb]{\smash{3}}}%
    \put(0.850942,0.042073){\color[rgb]{0,0,0}\makebox(0,0)[lb]{\smash{1}}}%
    \put(0,0){\includegraphics[width=\unitlength,page=2]{free_operad.pdf}}%
  \end{picture}%
\endgroup%

%% file: fig/tau.pdf_tex
%% Creator: Inkscape 0.91_64bit, www.inkscape.org
%% Accompanies image file 'tau.pdf' (pdf, eps, ps)
\begingroup%
  \makeatletter%
  \providecommand\color[2][]{%
    \errmessage{(Inkscape) Color is used for the text in Inkscape, but the package 'color.sty' is not loaded}%
    \renewcommand\color[2][]{}%
  }%
  \providecommand\transparent[1]{%
    \errmessage{(Inkscape) Transparency is used (non-zero) for the text in Inkscape, but the package 'transparent.sty' is not loaded}%
    \renewcommand\transparent[1]{}%
  }%
  \providecommand\rotatebox[2]{#2}%
  \ifx\svgwidth\undefined%
    \setlength{\unitlength}{70.250819bp}%
    \ifx\svgscale\undefined%
      \relax%
    \else%
      \setlength{\unitlength}{\unitlength * \real{\svgscale}}%
    \fi%
  \else%
    \setlength{\unitlength}{\svgwidth}%
  \fi%
  \global\let\svgwidth\undefined%
  \global\let\svgscale\undefined%
  \makeatother%
  \begin{picture}(1,1.200916)%
    \put(0,0){\includegraphics[width=\unitlength,page=1]{tau.pdf}}%
    \put(0.214745,1.084049){\color[rgb]{0,0,0}\makebox(0,0)[lb]{\smash{$1$}}}%
    \put(0.661715,1.079982){\color[rgb]{0,0,0}\makebox(0,0)[lb]{\smash{$2$}}}%
  \end{picture}%
\endgroup%

%% file: fig/phi.pdf_tex
%% Creator: Inkscape 0.91_64bit, www.inkscape.org
%% Accompanies image file 'phi.pdf' (pdf, eps, ps)
\begingroup%
  \makeatletter%
  \providecommand\color[2][]{%
    \errmessage{(Inkscape) Color is used for the text in Inkscape, but the package 'color.sty' is not loaded}%
    \renewcommand\color[2][]{}%
  }%
  \providecommand\transparent[1]{%
    \errmessage{(Inkscape) Transparency is used (non-zero) for the text in Inkscape, but the package 'transparent.sty' is not loaded}%
    \renewcommand\transparent[1]{}%
  }%
  \providecommand\rotatebox[2]{#2}%
  \ifx\svgwidth\undefined%
    \setlength{\unitlength}{70.469291bp}%
    \ifx\svgscale\undefined%
      \relax%
    \else%
      \setlength{\unitlength}{\unitlength * \real{\svgscale}}%
    \fi%
  \else%
    \setlength{\unitlength}{\svgwidth}%
  \fi%
  \global\let\svgwidth\undefined%
  \global\let\svgscale\undefined%
  \makeatother%
  \begin{picture}(1,1.180401)%
    \put(0,0){\includegraphics[width=\unitlength,page=1]{phi.pdf}}%
    \put(0.217391,1.063895){\color[rgb]{0,0,0}\makebox(0,0)[lb]{\smash{$1$}}}%
    \put(0.677454,1.061029){\color[rgb]{0,0,0}\makebox(0,0)[lb]{\smash{$2$}}}%
    \put(0,0){\includegraphics[width=\unitlength,page=2]{phi.pdf}}%
  \end{picture}%
\endgroup%

%% file: fig/psi.pdf_tex
%% Creator: Inkscape 0.91_64bit, www.inkscape.org
%% Accompanies image file 'psi.pdf' (pdf, eps, ps)
\begingroup%
  \makeatletter%
  \providecommand\color[2][]{%
    \errmessage{(Inkscape) Color is used for the text in Inkscape, but the package 'color.sty' is not loaded}%
    \renewcommand\color[2][]{}%
  }%
  \providecommand\transparent[1]{%
    \errmessage{(Inkscape) Transparency is used (non-zero) for the text in Inkscape, but the package 'transparent.sty' is not loaded}%
    \renewcommand\transparent[1]{}%
  }%
  \providecommand\rotatebox[2]{#2}%
  \ifx\svgwidth\undefined%
    \setlength{\unitlength}{70.469291bp}%
    \ifx\svgscale\undefined%
      \relax%
    \else%
      \setlength{\unitlength}{\unitlength * \real{\svgscale}}%
    \fi%
  \else%
    \setlength{\unitlength}{\svgwidth}%
  \fi%
  \global\let\svgwidth\undefined%
  \global\let\svgscale\undefined%
  \makeatother%
  \begin{picture}(1,1.198669)%
    \put(0,0){\includegraphics[width=\unitlength,page=1]{psi.pdf}}%
    \put(0.220648,1.080137){\color[rgb]{0,0,0}\makebox(0,0)[lb]{\smash{$1$}}}%
    \put(0.670692,1.082164){\color[rgb]{0,0,0}\makebox(0,0)[lb]{\smash{$2$}}}%
  \end{picture}%
\endgroup%

%% file: fig/alpha_c.pdf_tex
%% Creator: Inkscape 0.91_64bit, www.inkscape.org
%% Accompanies image file 'alpha_c.pdf' (pdf, eps, ps)
\begingroup%
  \makeatletter%
  \providecommand\color[2][]{%
    \errmessage{(Inkscape) Color is used for the text in Inkscape, but the package 'color.sty' is not loaded}%
    \renewcommand\color[2][]{}%
  }%
  \providecommand\transparent[1]{%
    \errmessage{(Inkscape) Transparency is used (non-zero) for the text in Inkscape, but the package 'transparent.sty' is not loaded}%
    \renewcommand\transparent[1]{}%
  }%
  \providecommand\rotatebox[2]{#2}%
  \ifx\svgwidth\undefined%
    \setlength{\unitlength}{102.250819bp}%
    \ifx\svgscale\undefined%
      \relax%
    \else%
      \setlength{\unitlength}{\unitlength * \real{\svgscale}}%
    \fi%
  \else%
    \setlength{\unitlength}{\svgwidth}%
  \fi%
  \global\let\svgwidth\undefined%
  \global\let\svgscale\undefined%
  \makeatother%
  \begin{picture}(1,0.536339)%
    \put(0,0){\includegraphics[width=\unitlength,page=1]{alpha_c.pdf}}%
    \put(0.13171,0.456046){\color[rgb]{0,0,0}\makebox(0,0)[lb]{\smash{$1$}}}%
    \put(0.367824,0.456046){\color[rgb]{0,0,0}\makebox(0,0)[lb]{\smash{$2$}}}%
    \put(0.722693,0.456046){\color[rgb]{0,0,0}\makebox(0,0)[lb]{\smash{$3$}}}%
    \put(0.271422,-0.038536){\color[rgb]{0,0,0}\makebox(0,0)[lb]{\smash{}}}%
  \end{picture}%
\endgroup%

%% file: fig/alpha_o.pdf_tex
%% Creator: Inkscape 0.91_64bit, www.inkscape.org
%% Accompanies image file 'alpha_o.pdf' (pdf, eps, ps)
\begingroup%
  \makeatletter%
  \providecommand\color[2][]{%
    \errmessage{(Inkscape) Color is used for the text in Inkscape, but the package 'color.sty' is not loaded}%
    \renewcommand\color[2][]{}%
  }%
  \providecommand\transparent[1]{%
    \errmessage{(Inkscape) Transparency is used (non-zero) for the text in Inkscape, but the package 'transparent.sty' is not loaded}%
    \renewcommand\transparent[1]{}%
  }%
  \providecommand\rotatebox[2]{#2}%
  \ifx\svgwidth\undefined%
    \setlength{\unitlength}{102.069291bp}%
    \ifx\svgscale\undefined%
      \relax%
    \else%
      \setlength{\unitlength}{\unitlength * \real{\svgscale}}%
    \fi%
  \else%
    \setlength{\unitlength}{\svgwidth}%
  \fi%
  \global\let\svgwidth\undefined%
  \global\let\svgscale\undefined%
  \makeatother%
  \begin{picture}(1,0.535526)%
    \put(0,0){\includegraphics[width=\unitlength,page=1]{alpha_o.pdf}}%
    \put(0.133501,0.45509){\color[rgb]{0,0,0}\makebox(0,0)[lb]{\smash{$1$}}}%
    \put(0.367236,0.45509){\color[rgb]{0,0,0}\makebox(0,0)[lb]{\smash{$2$}}}%
    \put(0.718538,0.45509){\color[rgb]{0,0,0}\makebox(0,0)[lb]{\smash{$3$}}}%
    \put(0.279061,-0.043171){\color[rgb]{0,0,0}\makebox(0,0)[lb]{\smash{}}}%
  \end{picture}%
\endgroup%

%% file: fig/exa_shuffle.pdf_tex
%% Creator: Inkscape 0.91_64bit, www.inkscape.org
%% Accompanies image file 'exa_shuffle.pdf' (pdf, eps, ps)
\begingroup%
  \makeatletter%
  \providecommand\color[2][]{%
    \errmessage{(Inkscape) Color is used for the text in Inkscape, but the package 'color.sty' is not loaded}%
    \renewcommand\color[2][]{}%
  }%
  \providecommand\transparent[1]{%
    \errmessage{(Inkscape) Transparency is used (non-zero) for the text in Inkscape, but the package 'transparent.sty' is not loaded}%
    \renewcommand\transparent[1]{}%
  }%
  \providecommand\rotatebox[2]{#2}%
  \ifx\svgwidth\undefined%
    \setlength{\unitlength}{166.753259bp}%
    \ifx\svgscale\undefined%
      \relax%
    \else%
      \setlength{\unitlength}{\unitlength * \real{\svgscale}}%
    \fi%
  \else%
    \setlength{\unitlength}{\svgwidth}%
  \fi%
  \global\let\svgwidth\undefined%
  \global\let\svgscale\undefined%
  \makeatother%
  \begin{picture}(1,0.387709)%
    \put(0,0){\includegraphics[width=\unitlength,page=1]{exa_shuffle.pdf}}%
  \end{picture}%
\endgroup%

%% file: fig/decomp_papb.pdf_tex
%% Creator: Inkscape 0.91_64bit, www.inkscape.org
%% Accompanies image file 'decomp_papb.pdf' (pdf, eps, ps)
\begingroup%
  \makeatletter%
  \providecommand\color[2][]{%
    \errmessage{(Inkscape) Color is used for the text in Inkscape, but the package 'color.sty' is not loaded}%
    \renewcommand\color[2][]{}%
  }%
  \providecommand\transparent[1]{%
    \errmessage{(Inkscape) Transparency is used (non-zero) for the text in Inkscape, but the package 'transparent.sty' is not loaded}%
    \renewcommand\transparent[1]{}%
  }%
  \providecommand\rotatebox[2]{#2}%
  \ifx\svgwidth\undefined%
    \setlength{\unitlength}{411.381352bp}%
    \ifx\svgscale\undefined%
      \relax%
    \else%
      \setlength{\unitlength}{\unitlength * \real{\svgscale}}%
    \fi%
  \else%
    \setlength{\unitlength}{\svgwidth}%
  \fi%
  \global\let\svgwidth\undefined%
  \global\let\svgscale\undefined%
  \makeatother%
  \begin{picture}(1,0.30986)%
    \put(0,0){\includegraphics[width=\unitlength,page=1]{decomp_papb.pdf}}%
    \put(0.406155,0.141207){\color[rgb]{0,0,0}\makebox(0,0)[lb]{\smash{$=$}}}%
    \put(0,0){\includegraphics[width=\unitlength,page=2]{decomp_papb.pdf}}%
    \put(0.864541,0.292892){\color[rgb]{0,0,0}\makebox(0,0)[lb]{\smash{$\scriptstyle x_1$}}}%
    \put(0.864541,0.204652){\color[rgb]{0,0,0}\makebox(0,0)[lb]{\smash{$\scriptstyle x_1'$}}}%
    \put(0.860652,0.108148){\color[rgb]{0,0,0}\makebox(0,0)[lb]{\smash{$\scriptstyle x_2'$}}}%
    \put(0.864541,0.00897){\color[rgb]{0,0,0}\makebox(0,0)[lb]{\smash{$\scriptstyle x_2$}}}%
    \put(0.864541,0.248164){\color[rgb]{0,0,0}\makebox(0,0)[lb]{\smash{$\scriptstyle \mu$}}}%
    \put(0.864541,0.054045){\color[rgb]{0,0,0}\makebox(0,0)[lb]{\smash{$\scriptstyle \mu'$}}}%
    \put(0.864541,0.150931){\color[rgb]{0,0,0}\makebox(0,0)[lb]{\smash{$\scriptstyle \mu_o(X', f(X''))$}}}%
  \end{picture}%
\endgroup%

%% file: fig/zeta.pdf_tex
%% Creator: Inkscape 0.91_64bit, www.inkscape.org
%% Accompanies image file 'zeta.pdf' (pdf, eps, ps)
\begingroup%
  \makeatletter%
  \providecommand\color[2][]{%
    \errmessage{(Inkscape) Color is used for the text in Inkscape, but the package 'color.sty' is not loaded}%
    \renewcommand\color[2][]{}%
  }%
  \providecommand\transparent[1]{%
    \errmessage{(Inkscape) Transparency is used (non-zero) for the text in Inkscape, but the package 'transparent.sty' is not loaded}%
    \renewcommand\transparent[1]{}%
  }%
  \providecommand\rotatebox[2]{#2}%
  \ifx\svgwidth\undefined%
    \setlength{\unitlength}{381.648456bp}%
    \ifx\svgscale\undefined%
      \relax%
    \else%
      \setlength{\unitlength}{\unitlength * \real{\svgscale}}%
    \fi%
  \else%
    \setlength{\unitlength}{\svgwidth}%
  \fi%
  \global\let\svgwidth\undefined%
  \global\let\svgscale\undefined%
  \makeatother%
  \begin{picture}(1,0.362813)%
    \put(0,0){\includegraphics[width=\unitlength,page=1]{zeta.pdf}}%
    \put(0.106736,0.153694){\color[rgb]{0,0,0}\makebox(0,0)[lb]{\smash{$\times$}}}%
    \put(0.304824,0.152646){\color[rgb]{0,0,0}\makebox(0,0)[lb]{\smash{$\times$}}}%
    \put(0.162285,0.266887){\color[rgb]{0,0,0}\makebox(0,0)[lb]{\smash{$2$}}}%
    \put(0.204208,0.266887){\color[rgb]{0,0,0}\makebox(0,0)[lb]{\smash{$1$}}}%
    \put(0.246131,0.266887){\color[rgb]{0,0,0}\makebox(0,0)[lb]{\smash{$3$}}}%
    \put(0.021841,0.261647){\color[rgb]{0,0,0}\makebox(0,0)[lb]{\smash{$1$}}}%
    \put(0.065861,0.261647){\color[rgb]{0,0,0}\makebox(0,0)[lb]{\smash{$2$}}}%
    \put(0.353036,0.267935){\color[rgb]{0,0,0}\makebox(0,0)[lb]{\smash{$1$}}}%
    \put(0.39496,0.267935){\color[rgb]{0,0,0}\makebox(0,0)[lb]{\smash{$2$}}}%
    \put(0.434787,0.267935){\color[rgb]{0,0,0}\makebox(0,0)[lb]{\smash{$3$}}}%
    \put(0.478806,0.267935){\color[rgb]{0,0,0}\makebox(0,0)[lb]{\smash{$1$}}}%
    \put(0.52073,0.263743){\color[rgb]{0,0,0}\makebox(0,0)[lb]{\smash{$2$}}}%
    \put(0.791136,0.341301){\color[rgb]{0,0,0}\makebox(0,0)[lb]{\smash{$2$}}}%
    \put(0.833059,0.341301){\color[rgb]{0,0,0}\makebox(0,0)[lb]{\smash{$2$}}}%
    \put(0.877079,0.341301){\color[rgb]{0,0,0}\makebox(0,0)[lb]{\smash{$3$}}}%
    \put(0.916906,0.341301){\color[rgb]{0,0,0}\makebox(0,0)[lb]{\smash{$1$}}}%
    \put(0.963022,0.341301){\color[rgb]{0,0,0}\makebox(0,0)[lb]{\smash{$1$}}}%
  \end{picture}%
\endgroup%

%% file: fig/rho.pdf_tex
%% Creator: Inkscape 0.91_64bit, www.inkscape.org
%% Accompanies image file 'rho.pdf' (pdf, eps, ps)
\begingroup%
  \makeatletter%
  \providecommand\color[2][]{%
    \errmessage{(Inkscape) Color is used for the text in Inkscape, but the package 'color.sty' is not loaded}%
    \renewcommand\color[2][]{}%
  }%
  \providecommand\transparent[1]{%
    \errmessage{(Inkscape) Transparency is used (non-zero) for the text in Inkscape, but the package 'transparent.sty' is not loaded}%
    \renewcommand\transparent[1]{}%
  }%
  \providecommand\rotatebox[2]{#2}%
  \ifx\svgwidth\undefined%
    \setlength{\unitlength}{263.483618bp}%
    \ifx\svgscale\undefined%
      \relax%
    \else%
      \setlength{\unitlength}{\unitlength * \real{\svgscale}}%
    \fi%
  \else%
    \setlength{\unitlength}{\svgwidth}%
  \fi%
  \global\let\svgwidth\undefined%
  \global\let\svgscale\undefined%
  \makeatother%
  \begin{picture}(1,0.323799)%
    \put(0,0){\includegraphics[width=\unitlength,page=1]{rho.pdf}}%
    \put(0.14898,0.140785){\color[rgb]{0,0,0}\makebox(0,0)[lb]{\smash{$\times$}}}%
    \put(0.376698,0.141706){\color[rgb]{0,0,0}\makebox(0,0)[lb]{\smash{$\times$}}}%
    \put(0,0){\includegraphics[width=\unitlength,page=2]{rho.pdf}}%
    \put(0.02753,0.247054){\color[rgb]{0,0,0}\makebox(0,0)[lb]{\smash{$1$}}}%
    \put(0.097363,0.244018){\color[rgb]{0,0,0}\makebox(0,0)[lb]{\smash{$2$}}}%
    \put(0.222771,0.244907){\color[rgb]{0,0,0}\makebox(0,0)[lb]{\smash{$1$}}}%
    \put(0.301476,0.244907){\color[rgb]{0,0,0}\makebox(0,0)[lb]{\smash{$2$}}}%
    \put(0.3387,0.244907){\color[rgb]{0,0,0}\makebox(0,0)[lb]{\smash{$3$}}}%
    \put(0.443495,0.244907){\color[rgb]{0,0,0}\makebox(0,0)[lb]{\smash{$1$}}}%
    \put(0.482966,0.244907){\color[rgb]{0,0,0}\makebox(0,0)[lb]{\smash{$2$}}}%
    \put(0.510292,0.244907){\color[rgb]{0,0,0}\makebox(0,0)[lb]{\smash{$3$}}}%
    \put(0.546727,0.244907){\color[rgb]{0,0,0}\makebox(0,0)[lb]{\smash{$1$}}}%
    \put(0.60138,0.244907){\color[rgb]{0,0,0}\makebox(0,0)[lb]{\smash{$2$}}}%
    \put(0.792663,0.301706){\color[rgb]{0,0,0}\makebox(0,0)[lb]{\smash{$1$}}}%
    \put(0.83517,0.301706){\color[rgb]{0,0,0}\makebox(0,0)[lb]{\smash{$2$}}}%
    \put(0.85946,0.301706){\color[rgb]{0,0,0}\makebox(0,0)[lb]{\smash{$3$}}}%
    \put(0.898931,0.301706){\color[rgb]{0,0,0}\makebox(0,0)[lb]{\smash{$1^*$}}}%
    \put(0.950547,0.301706){\color[rgb]{0,0,0}\makebox(0,0)[lb]{\smash{$2^*$}}}%
  \end{picture}%
\endgroup%

%% file: fig/rho_phi.pdf_tex
%% Creator: Inkscape 0.91_64bit, www.inkscape.org
%% Accompanies image file 'rho_phi.pdf' (pdf, eps, ps)
\begingroup%
  \makeatletter%
  \providecommand\color[2][]{%
    \errmessage{(Inkscape) Color is used for the text in Inkscape, but the package 'color.sty' is not loaded}%
    \renewcommand\color[2][]{}%
  }%
  \providecommand\transparent[1]{%
    \errmessage{(Inkscape) Transparency is used (non-zero) for the text in Inkscape, but the package 'transparent.sty' is not loaded}%
    \renewcommand\transparent[1]{}%
  }%
  \providecommand\rotatebox[2]{#2}%
  \ifx\svgwidth\undefined%
    \setlength{\unitlength}{63.95003bp}%
    \ifx\svgscale\undefined%
      \relax%
    \else%
      \setlength{\unitlength}{\unitlength * \real{\svgscale}}%
    \fi%
  \else%
    \setlength{\unitlength}{\svgwidth}%
  \fi%
  \global\let\svgwidth\undefined%
  \global\let\svgscale\undefined%
  \makeatother%
  \begin{picture}(1,1.346806)%
    \put(0,0){\includegraphics[width=\unitlength,page=1]{rho_phi.pdf}}%
    \put(0.070075,0.901934){\color[rgb]{0,0,0}\makebox(0,0)[lb]{\smash{$\phi$}}}%
    \put(0.070075,0.130589){\color[rgb]{0,0,0}\makebox(0,0)[lb]{\smash{$\phi$}}}%
    \put(0.146857,1.255782){\color[rgb]{0,0,0}\makebox(0,0)[lb]{\smash{$2$}}}%
    \put(0.30725,1.255782){\color[rgb]{0,0,0}\makebox(0,0)[lb]{\smash{$1$}}}%
    \put(0.41336,1.255782){\color[rgb]{0,0,0}\makebox(0,0)[lb]{\smash{$3$}}}%
    \put(0.565487,1.255782){\color[rgb]{0,0,0}\makebox(0,0)[lb]{\smash{$1^*$}}}%
    \put(0.796248,1.255782){\color[rgb]{0,0,0}\makebox(0,0)[lb]{\smash{$2^*$}}}%
    \put(0,0){\includegraphics[width=\unitlength,page=2]{rho_phi.pdf}}%
    \put(0.070299,0.516942){\color[rgb]{0,0,0}\makebox(0,0)[lb]{\smash{$\phi$}}}%
  \end{picture}%
\endgroup%